\documentclass
{amsproc}
\usepackage[english]{babel}
\usepackage{amsmath}
\usepackage{amsthm}
\usepackage{amssymb}
\usepackage{graphicx}
\usepackage{multirow}
\usepackage{adjustbox}
\usepackage{caption}
\usepackage{subfig}
\usepackage{enumitem}

\makeatletter
\renewcommand*\env@matrix[1][*\c@MaxMatrixCols c]{%
  \hskip -\arraycolsep
  \let\@ifnextchar\new@ifnextchar
  \array{#1}}
\makeatother

\DeclareGraphicsExtensions{.png,.pdf,.eps}

\usepackage[all,2cell]{xy}
\usepackage{tikz}
\usepackage{pgf}
\usepackage{rotating}
\usepackage{enumerate}
\usetikzlibrary{cd}

\usepackage{array}
\newcolumntype{H}{>{\setbox0=\hbox\bgroup}c<{\egroup}@{}}
\newcolumntype{L}{>{$}l<{$}} 
\newcolumntype{C}{>{$}c<{$}} 

\usepackage{hyperref}

\newtheorem*{theorem*}{Theorem}
\newtheorem*{corollary*}{Corollary}
\newtheorem{theorem}{Theorem}[section]
\newtheorem{lemma}[theorem]{Lemma}
\newtheorem{corollary}[theorem]{Corollary}
\newtheorem{proposition}[theorem]{Proposition}

\theoremstyle{definition}

\newtheorem{definition}[theorem]{Definition}
\newtheorem{notation}[theorem]{Notation}

\newtheorem{remark}[theorem]{Remark}

\newtheorem{set}[theorem]{Setup}
\newtheorem{example}[theorem]{Example}

\title[Inversion maps and torus actions on RH varieties]{Inversion maps and torus actions on rational homogeneous varieties}

\author[Franceschini]{Alberto Franceschini}
\address{Dipartimento di Matematica, Universit\`a degli Studi di Trento, via
Sommarive 14 I-38123 Povo di Trento (TN), Italy}
 
\email{alberto.franceschini@unitn.it, eduardo.solaconde@unitn.it}

\author[Sol\'a Conde]{Luis E. Sol\'a Conde}

\subjclass[2010]{Primary 14L30; Secondary 14E30, 14L24, 14M17}

\thanks{First author supported by the PhD program in Mathematics of the University of Trento. Second author is supported by PRIN project ``Geometria delle variet\`a algebriche''. }

\def\C{{\mathbb C}}

\def\P{{\mathbb P}}
\def\Q{{\mathbb Q}}
\def\R{{\mathbb R}}
\def\Z{{\mathbb Z}}

\def\cB{{\mathcal B}}
\def\cC{{\mathcal C}}
\def\cD{{\mathcal D}}
\def\cE{{\mathcal E}}

\def\cG{{\mathcal{G}}}

\def\cK{{\mathcal K}}

\def\cN{{\mathcal{N}}}
\def\cO{{\mathcal{O}}}

\def\cQ{{\mathcal{Q}}}

\def\cS{{\mathcal S}}

\def\cY{{\mathcal Y}}
\def\Q{{\mathbb{Q}}}

\def\fg{{\mathfrak g}}
\def\fh{{\mathfrak h}}
\def\fp{{\mathfrak p}}
\def\fb{{\mathfrak b}}

\def\DA{{\rm A}}
\def\DB{{\rm B}}
\def\DC{{\rm C}}
\def\DD{{\rm D}}
\def\DE{{\rm E}}

\DeclareMathOperator{\HH}{H}

\def\operatorname#1{\mathop{\rm #1}\nolimits}

\def\Proj{\operatorname{Proj}}

\def\Pic{\operatorname{Pic}}
\def\Hom{\operatorname{Hom}}

\def\id{\operatorname{id}}

\def\rank{\operatorname{rank}}

\def\conj{\operatorname{conj}}

\def\Mov{{\operatorname{Mov}}}

\def\ex{{\operatorname{ex}}}

\def\Ad{\operatorname{Ad}}

\def\PGL{\operatorname{PGL}}

\def\Sec{\operatorname{Sec}}

\DeclareMathOperator{\No}{N}

\newcommand{\pb}{\ar@{}[dr]|{\text{\pigpenfont J}}}
\def\Mo{\operatorname{\hspace{0cm}M}}

\def\point{\operatorname{pt}}

\renewcommand{\epsilon}{\varepsilon}
\renewcommand{\phi}{\varphi}

\newcommand\ra{{\ \rightarrow\ }}
\newcommand\lra{\longrightarrow}

\newcommand\ignore[1]{}

\newcommand{\shse}[3]{0 ~\ra ~#1~ \lra ~#2~ \lra ~#3~ \ra~ 0}
\makeatletter
\newcommand{\xleftrightarrow}[2][]{\ext@arrow 3359\leftrightarrowfill@{#1}{#2}}
\makeatother
\newcommand{\xdasharrow}[2][->]{
\tikz[baseline=-\the\dimexpr\fontdimen22\textfont2\relax]{
\node[anchor=south,font=\scriptsize, inner ysep=1.5pt,outer xsep=2.2pt](x){#2};
\draw[shorten <=3.4pt,shorten >=3.4pt,dashed,#1](x.south west)--(x.south east);
}}

\begin{document}
\begin{abstract}
Complex projective algebraic varieties with $\C^*$-actions can be thought of as geometric counterparts of birational transformations. In this paper we describe geometrically the birational transformations associated to rational homogeneous varieties endowed with a $\C^*$-action with no proper isotropy subgroups.
\end{abstract}
\maketitle
\tableofcontents


\section{Introduction}\label{sec:intro}

The transition from classic projective to modern birational geometry in the 1980s left algebraic geometry with the task of understanding brand new birational transformations such as flips and flops. 
An interesting idea, that sprung out of the work of Reid \cite{ReidToric}, Thaddeus \cite{Thaddeus1996,Thaddeus} and W{\l}odarczyk \cite{Wlodarczyk}, is that these transformations may be understood by means of actions of tori: given a birational map $\psi$ between complex projective varieties, there exists an algebraic variety $X$ (not proper, in general) admitting a $\C^*$-action such that $\psi$ is the map induced among two GIT-quotients of $X$ by $\C^*$. 

On the other hand, if one starts with a projective variety $X$ admitting a nontrivial $\C^*$-action, then one may consider the geometric quotients of the action and the induced birational map among them. 
Among these quotients, two distinguished ones are particularly relevant: starting with the general point $x\in X$, we call \emph{sink} and \emph{source} of the action the (unique) fixed-point components $Y_-,Y_+$ that contain, respectively, $\lim_{t\to 0}t^{-1}\cdot x$ and  $\lim_{t\to 0}t\cdot x$. Considering the {\em Bia{\l}ynicki-Birula cells},
\[
X^{\pm}(Y_{\pm}):=\left\{x\in X: \lim_{t\to 0}t^{\pm 1} \cdot x\in Y_{\pm}\right\},
\]
the quotients $\cG_{\pm}:=\left(X^{\pm}(Y_{\pm})\setminus Y_{\pm}\right)/\C^*$ are geometric quotients of $X$ and, in the case in which $X$ is smooth  (see \cite{BB,kon}), $\cG_{\pm}$ are weighted projective bundles over $Y_{\pm}$.
The nonempty intersection $X^{-}(Y_{-})\cap X^{+}(Y_{+})$ defines a birational map $\psi_a:X^{-}(Y_{-})\dashrightarrow  X^{+}(Y_{+})$ that descends via the quotient by $\C^*$ to a birational map
\[
\psi:\cG_-\dashrightarrow \cG_+,
\]
that we call the {\em birational map induced by the $\C^*$-action}. This map encodes many geometric properties of $X$; for instance, it is equivariant with respect to the action of the centralizer of $\C^*$ in the group of automorphisms of $X$. 
When the group of automorphisms of $X$ is large (for instance if $X$ is rational homogeneous), the list of possible associated birational maps is expected to be small.

Note that when $X$ is smooth and the $\C^*$-action has no finite stabilizers (in our language we say that the action is \emph{equalized}), $\cG_\pm$ are smooth varieties. 
One can describe $\psi$ as a sequence of Atiyah flips, blowups and blowdowns (cf. \cite{WORS3}). 

The goal of this paper is to study equalized $\C^*$-actions on some rational homogeneous varieties (\emph{RH variety}, for short) of Picard number one $G/P$ ($G$ semisimple algebraic group, $P$ maximal parabolic subgroup) and to describe their associated birational maps. 
A key property of these transformations is that they are equivariant with respect to a certain reductive subgroup $G_0\subset G$. 

\begin{example}
The simplest example is the projectivization of the inversion map of matrices $$
\psi:\P(M_{n\times n})\dashrightarrow \P(M_{n\times n}),\qquad [A]\longmapsto [A^{-1}],
$$ 
which is $\PGL(n)$-equivariant.  
It was already noted by Thaddeus in \cite[Section 4]{Thaddeus} that this birational map is induced by a certain $\C^*$-action on $X$, the Grassmannian of $n$-dimensional subspaces of a $2n$-dimensional complex vector space. The restriction of this action to a certain $\C^*$-invariant subvariety $(\P^1)^n \subset X$ has the standard Cremona transformation 
\[
\psi:\P^{n-1}\dashrightarrow\P^{n-1},\qquad \left(x_1:\ldots:x_{n}\right)\longmapsto\left(x_1^{-1}:\ldots:x_n^{-1}\right), 
\]
as the associated birational map. This map, which can be thought of as the projectivization of the inversion of diagonal $n\times n$ matrices, is equivariant with respect to the action of an algebraic torus $(\C^*)^{n-1}$.
\end{example}

Thaddeus' observation -- and its analogous formulation in the case of the inversion of symmetric matrices, cf. \cite[Lemma 3.6]{MMW} -- can be extended to other $\C^*$-actions on rational homogeneous varieties. Roughly speaking, we expect the associated birational maps to behave like ``inversion maps'' for certain algebraic structures on $X^{\pm}(Y_{\pm})$. 

This idea is particularly clear in the case in which the action is equalized and the extremal fixed-point components $Y_{\pm}$ are isolated points, i.e. in the case that the induced birational map is a Cremona transformation:
$$
\psi:\cG_-=\P(T_{X,Y_-})\dashrightarrow \P(T_{X,Y_+})=\cG_+.
$$
In this situation we will prove the following:

\begin{theorem}\label{thm:Jordan}
Consider an equalized $\C^*$-action on an RH variety $X$ of Picard number one with isolated sink and source, let $\psi:\P(T_{X,Y_-})\dashrightarrow \P(T_{X,Y_+})$ be the Cremona transformation induced by the action. Then there exists a unique structure of Jordan algebra on $T_{X,Y_-}$ and a linear isomorphism $\alpha:T_{X,Y_+}\to T_{X,Y_-}$, such that $\jmath :=\P(\alpha)\circ \psi$ is the projectivization of the corresponding inversion map in $T_{X,Y_-}$.
\end{theorem}

The precise way in which $\psi$ can be constructed upon the map $\jmath$ will be described in Section \ref{sec:j-str}. 
Our result can be considered a reformulation of \cite[2.21]{Spr73} into the language of $\C^*$-actions.
 Note that in the same text the author shows how the inverse map completely determines the structure of Jordan algebra and that any simple Jordan algebra appears this way. 
 A key point in the proof of Theorem \ref{thm:Jordan} will be the fact that the $\C^*$-actions considered in the statement are completely determined by a grading on the Lie algebra $\fg$ of $G$; the hypotheses of Theorem \ref{thm:Jordan} can be rephrased by saying that the grading is \emph{short} (see Section \ref{sssec:shortgr}) and \emph{balanced} (see Section \ref{sec:j-str} for the precise definition). 

Given a short and balanced grading on a simple Lie algebra $\fg$, one may consider the actions induced on the other RH $G$-varieties of Picard number one (not only those on which the sink and the source are isolated) and study the induced birational maps. 
As in Theorem \ref{thm:Jordan}, up to composition with a biregular map, we obtain a birational involution of the corresponding geometric quotient $\cG_-$, which in this case is the projectivization of a homogeneous vector bundle over an RH variety.
When $G$ is of classical type, we study explicit descriptions of the birational map $\psi$, obtaining the following (we refer to Section \ref{ssec:notclassic} for the notations):   

\begin{theorem}\label{thm:unbalanced}
	Let $G$ be a simple algebraic group whose Lie algebra $\fg$ is of classical type and admits a short and balanced grading induced by $\sigma$. Let $X$ be an RH $G$-variety of Picard number one endowed with the $\C^*$-action determined by $\sigma$. Assume that the extremal fixed point components of the action are not isolated points. 
	Then the induced birational transformation $\psi:\cG_- \dashrightarrow \cG_+$ is a small $\Q$-factorial modification, uniquely determined by the geometric quotients $\cG_\pm$ of $X$ (up to automorphisms of $\cG_\pm$). Their complete list is the following:
\begin{itemize}[leftmargin=*]
\item $X = \DA_{2n-1}(k)$ is the Grassmannian of $k$-dimensional linear subspaces in a $2n$-dimensional vector space $V$, $k \le n$, endowed with a $\C^*$-action producing a weight  decomposition $V=V_- \oplus V_+$ with $\dim V_\pm=n$. The extremal fixed point components $Y_\pm$  are Grassmannians $\DA_{n-1}(k)$ of $k$-dimensional linear subspaces of $V_\pm$, and $\cG_\pm\simeq \P\big(\cS_\pm^\vee \otimes V_\mp \big)$, where $\cS_\pm$ denote the corresponding universal bundles (of rank $k$) on $Y_\pm$.

\item $X= \DC_n(k)$ (resp. $X \simeq \DD_n(k)$ with $k \le n-2$ or $X \simeq \DD_n(n)$, $n$ even) is the symplectic (resp. orthogonal) Grassmannian of isotropic subspaces in a $2n$-dimensional vector space $V$, endowed with a $\C^*$-action producing a weight decomposition $V=V_- \oplus V_+$ with $\dim V_\pm=n$, $V_\pm$ isotropic. Then $Y_\pm$ are Grassmannians $\DA_{n-1}(k)$ of $k$-dimensional linear subspaces of $V_\pm$ and $\cG_\pm\simeq\P\big(\cN_{Y_-|X}\big)$, where 
		the normal bundles $\cN_{Y_\pm|X}$ are non-trivial extensions of
			\[
			\shse{\left( \cS_\pm \otimes \cQ_\pm \right)^\vee}{\cN_{Y_\pm|X}}{\cC_\pm},
			\]
			and, being $\cS_\pm,\cQ_\pm$ the corresponding universal bundles (of rank $k$ and $n-k$, respectively) on $Y_\pm$,
			\[\cC_\pm:=\begin{cases}
				S^2 \cS_\pm^\vee & \text{if }X \simeq \DC_n(k), \\
				\wedge^2 \cS_\pm^\vee & \text{if }X \simeq \DD_n(k).
				\end{cases}
			\]	

\item $X = \DD_n(n-1)$ is the orthogonal Grassmannian of isotropic subspaces in a $2n$-dimensional vector space $V$ ($n$ even), endowed with a $\C^*$-action producing a weight decomposition $V=V_- \oplus V_+$ with $\dim V_\pm=n$, $V_\pm$ isotropic. Then $Y_\pm\simeq \DA_{n-1}(n-1)$ are projective spaces parametrizing hyperplanes in $V_\pm$, and $\cG_\pm\simeq \P\big(\bigwedge^2 T_{Y_\pm}(-2)\big)$.

\item $X= \DB_n(k)$ (resp. $X= \DD_n(k)$) is the orthogonal Grassmannian of isotropic subspaces of a $(2n+1)$-dimensional (resp. $2n$-dimensional) vector space $V$, endowed with a $\C^*$-action producing a weight decomposition $V=V_-\oplus V_0\oplus V_+$ (corresponding to weights $-1,0,1$, respectively), $\dim(V_\pm)=1$. Then $Y_\pm \simeq \DB_{n-1}(k-1)$ (resp. $Y_\pm \simeq\DD_{n-1}(k-1)$ if $k \le n-2$, $(Y_-,Y_+)\simeq (\DD_{n-1}(n-2),\DD_{n-1}(n-1))$ if $k=n-1$, and $(Y_-,Y_+)\simeq (\DD_{n-1}(n-1),\DD_{n-1}(n-2))$ if $k=n$) are orthogonal Grassmannians of isotropic subspaces in $V_0$, and $\cG_\pm\simeq\P\big(\cQ_\pm\big)$, where $\cQ_\pm$ are the corresponding universal bundles on $Y_\pm$.
\end{itemize}
\end{theorem}

Besides the cases of simple Lie algebras of classical type, only the exceptional Lie algebra $\mathfrak e_7$ admits a short and balanced grading, that we will later denote by $\sigma_7$. The birational maps $\psi:\cG_- \dashrightarrow \cG_+$ induced by the corresponding $\C^*$-action on the RH varieties $\DE_7(k)$ can still be described in terms of certain universal bundles. 
We refer to Section \ref{ssec:E7} for the precise definitions of the vector bundles involved.

\begin{theorem}\label{thm:E7}
	Consider the short and balanced grading $\sigma_7$ on the Lie algebra $\mathfrak e_7$. Then the corresponding $\C^*$-action $H_7$ on the RH variety with Picard number one $\DE_7(k)$ is such that $Y_- \simeq \DE_6(k)$, $Y_+ \simeq \DE_6(s(k))$ (where $s(k)$ is symmetric node for the non-trivial automorphism of the Dynkin diagram $\DE_6$ and $s(7)=0$) and 
	\[
	\psi: \P \left( \cQ_k \right) \dashrightarrow \P \left( \cQ'_{s(k)} \right),
	\] 
	where $\cQ_k,\cQ'_{s(k)}$ are the quotient bundles over $Y_\pm$ arising, respectively, from the short exact sequences
	\begin{align*}
	\shse{\cS_k}{V(\omega_6) &\otimes \cO_{Y_-}}{\cQ_k},\\
	\shse{\cS'_{s(k)}}{V(\omega_1) &\otimes \cO_{Y_+}}{\cQ'_{s(k)}}.
	\end{align*}
\end{theorem}

The contents of Theorems \ref{thm:unbalanced} and \ref{thm:E7} are summarized in Table \ref{tab:summary} below.

\begin{table}[ht!]
 \centering
 \begin{adjustbox}{angle=90}
 \renewcommand\arraystretch{2.3}
 \begin{tabular}{|c|c|c|c|c|c|c|}
 	\hline
 	$X$ & conditions & $\sigma$ & $Y_-$ & $\cG_-$ & $Y_+$ & $\cG_+$ \\
 	\hline
 	\hline
 	$\DA_{2n-1}(k)$ & $k \le n$ & $\sigma_n$ & $\DA_{n-1}(k)$ & $\P \left( \cS^\vee_- \otimes V_+\right)$ & $\DA_{n-1}(k)$ & $\P \left( \cS^\vee_+ \otimes V_-\right)$ \\
 	\hline
 	\hline
 	$\DB_{n}(k)$ &  & $\sigma_1$ & $\DB_{n-1}(k-1)$ & $\P\left(\cQ_-\right)$ & $\DB_{n-1}(k-1)$ & $\P \left(\cQ_+\right)$ \\
 	\hline
 	\hline
 	$\DC_n(k)$ & $k \le n-1$ & $\sigma_n$ & $\DA_{n-1}(k)$ & Proposition \ref{prop:normalHn}  
 	& $\DA_{n-1}(k)$ & Proposition \ref{prop:normalHn} 
 	\\
 	\hline
 	\hline
 	$\DC_n(n)$ &  & $\sigma_n$ & $\point$ & $\P \left( S^2 V_+\right)$ & $\point$ & $\P \left( S^2 V_-\right)$ \\
 	\hline
 	\hline
 	\multirow{2}*{$\DD_n(k)$} & \multirow{2}*{$k \le n-2$} & $\sigma_1$ & $\DD_{n-1}(k-1)$ & $\P\left(\cQ_-\right)$ & $\DD_{n-1}(k-1)$ & $\P \left(\cQ_+\right)$ \\
 	\cline{3-7}
 	& & $\sigma_n$ & $\DA_{n-1}(k)$ & Proposition \ref{prop:normalHn} 	& $\DA_{n-1}(k)$ & Proposition \ref{prop:normalHn}
 	\\
 	\hline
 	\hline
 	\multirow{2}*{$\DD_n(n-1)$} & & $\sigma_1$ & $\DD_{n-1}(n-2)$ & $\P\left(\cQ_-\right)$ & $\DD_{n-1}(n-1)$ & $\P \left(\cQ_+\right)$ \\
 	\cline{2-7}
 	& $n$ even & $\sigma_n$ & $\DA_{n-1}(n-1)$ & $\P\left(\bigwedge^2 T_{\DA_{n-1}(n-1)}(-2)\right)$ & $\DA_{n-1}(n-1)$ & $\P \left(\bigwedge^2 T_{\DA_{n-1}(n-1)}(-2)\right)$ \\
 	\hline
 	\hline
 	 \multirow{2}*{$\DD_n(n)$} & & $\sigma_1$ & $\DD_{n-1}(n-1)$ & $\P\left(\cQ_-\right)$ & $\DD_{n-1}(n-2)$ & $\P \left(\cQ_+\right)$ \\
 	\cline{2-7}
 	& $n$ even & $\sigma_n$ & $\point$ & $\P \left( \bigwedge^2 V_+\right)$ & $\point$ & $\P \left( \bigwedge^2 V_- \right)$ \\
 	\hline
 	\hline
 	$\DE_7(k)$ & & $\sigma_7$ & $\DE_6(k)$ & $\P \left( \cQ_k \right)$ & $\DE_6(s(k))$ & $\P \left( \cQ_{s(k)}'\right)$ \\
 	\hline
 \end{tabular}
 \end{adjustbox}
 \caption{Birational transformations induced by a short and balanced grading on an RH variety $X$ of Picard number one.\label{tab:summary}}
 \end{table}

\subsection*{Outline}
The structure of the paper is the following. We start with a section on background material on $\C^*$-actions on RH varieties (Section \ref{sec:prelim}). Then in Section \ref{sec:equalGrass} we study the fixed-point components of an equalized $\C^*$-action on an RH variety of classical type via projective geometry.
In Section \ref{sec:j-str} we study Cremona transformations induced by $\C^*$-actions on RH varieties, proving Theorem \ref{thm:Jordan}.
Finally, Section \ref{sec:proofThm1.2} is devoted to the proof of Theorem \ref{thm:unbalanced} and Theorem \ref{thm:E7}.

\subsection*{Acknowledgements} The first author would like to thank J. Wi\'sniewski and the participants of the working group meeting (Sopot, Poland, October 15--17, 2021) of the project ``Algebraic Torus Action: Geometry and Combinatorics'', financed by Polish National Science Center, grant 2016/23/G/ST1/04282, and Deutsche Forschungsgemeinschaft, grant HA 4383/8, for interesting discussions and suggestions. 



\section{Preliminaries}\label{sec:prelim}


Throughout this paper, unless otherwise stated, all the varieties will be projective, smooth and defined over the field of complex numbers. Given a vector bundle $\cE$ over such a variety, we denote by $\P(\cE)$ its homothetical projectivization, that is
\[
\P(\cE):=\Proj\left(\bigoplus_{m \ge 0} S^m \cE^\vee \right).
\]

\subsection{Notation and basic facts on $\C^*$-actions}\label{ssec:basic}

Let $X$ be a smooth projective variety endowed with a nontrivial $\C^*$-action. We will follow the notations and conventions in \cite{WORS1}.

\begin{itemize}[leftmargin=*]
\item We will denote by $X^{\C^*}$ the set of fixed-points of the action, and $\cY$ the set of irreducible components of $X^{\C^*}$.
\item Given a point $x \in X$, we denote by $x_\pm:=\lim_{t\to 0}t^{\pm 1} \cdot x\in X^{\C^*}$ the \emph{sink} $x_-$ and the \emph{source} $x_+$ of the orbit $\C^* \cdot x$.
\item The only components $Y_-,Y_+\in \cY$ such that, for a general $x\in X$, $\lim_{t\to 0}t^{\pm 1}\cdot x\in Y_\pm$ are called \emph{sink} and \emph{source} of the action, respectively.
\item Given $Y\in\cY$ -- which is a smooth subvariety of $X$ -- the normal bundle of $Y$ in $X$ will be denoted $\cN_{Y|X}$. 
	It decomposes as a direct sum of two subbundles
	\[
	\cN_{Y|X}=\cN^-(Y)\oplus\cN^+(Y),
	\]
	on which $\C^*$ acts with negative and positive weights, respectively; their ranks are denoted by $\nu^\pm(Y)$. 
	The two summands $\cN^\pm(Y)$ are $\C^*$-equivariantly isomorphic to subsets of $X$, the so-called \emph{Bia{\l}ynicki-Birula cells} $X^{\pm}(Y)\subset X$, defined as
	\[
	X^{\pm}(Y):=\left\{x\in X:\lim_{t\to 0}t^{\pm 1}\cdot x\in Y\right\}.
	\]
	Note that $X^\pm(Y_\pm)$  are open subsets in $X$.
\item The action is \emph{equalized} if the weights of the action of $\C^*$ on $\cN_{Y|X}$ are equal to $\pm 1$ for every $Y\in \cY$; equivalently, the action has no nontrivial isotropy subgroups.
\item We get a birational map  $\psi_a:\cN^-(Y_-)\dashrightarrow \cN^+(Y_+)$
	defined as the composition 
	\[
	\cN^-(Y_-)\simeq X^-(Y_-)\hookleftarrow X^-(Y_-)\cap X^+(Y_+)\hookrightarrow X^+(Y_+)\simeq \cN^+(Y_+).
	\]
	If the $\C^*$-action is equalized, the map descends via the quotient by homotheties to a birational map 
	\begin{equation}\label{eq:birmap}
	\psi:\P\left(\cN_{Y_-|X}\right)\dashrightarrow \P\left(\cN_{Y_+|X}\right).
	\end{equation} 
	Note that this map sends a general point of $\P\left(\cN_{Y_-|X}\right)$ corresponding to the tangent direction at $x_-\in Y_-$ of the closure of an orbit $\C^* \cdot x$ ($x\in X$ general) to its tangent direction at $x_+\in Y_+$. 
	The maps $\psi_a$ and (in the equalized case) $\psi$ are called the \emph{birational maps associated with the $\C^*$-action on $X$}. 
\item Given an ample line bundle $L$ on $X$, a linearization of the $\C^*$-action on $L$ exists. 
	Then $\C^*$ acts on the fibers of $L|_{Y}$ by multiplication with a character, which we denote by $\mu_L(Y)$. 
	Up to multiplication with a character, we may assume that $\mu_L(Y_-)=0$; then it follows that $\mu_L(Y)>0$ for every $Y\in\cY\setminus\{Y_-\}$ and that the maximum value $\delta$ of $\mu_L$ is achieved at the source $Y_+$. We set:
	\[
	Y_r:=\bigsqcup_{\mu_L(Y)=r} Y.
	\]
	Denoting $V:=H^0(X,L)^\vee$, the $\C^*$-action induces a weight decomposition 
	\[
		V=\bigoplus_{r\in \Z} V_r.
	\]
	If $L$ is very ample, we may write $Y_r=\P(V_r)\cap X$ for every $r$.

\item Note that $\mu_{rL}=r\mu_L$. 
	We will be interested in the case in which the Picard number of $X$ is one and we will assort the fixed-point components of the $\C^*$-action by their weights with respect to the ample generator of the Picard group of the variety.
\item Given the closure $C$ of an orbit $\C^* \cdot x$ satisfying that $x_-\in Y$, $x_+\in Y'$, it follows that 
	\[
	C\cdot L=(\mu_L(Y')-\mu_L(Y))t,
	\]
	where $t$ is the weight of $\C^*$ on $T_{C,x_+}$; if the action is equalized, then $t=1$ for every $x$. 
\end{itemize}

\begin{remark}\label{rem:WORS3}
In the case in which the $\C^*$-action is equalized and the Picard number of $X$ is equal to one, which will be our main interest, the map $\psi$ may be decomposed as a sequence of Atiyah flips and, in certain situations, a blowup and a blowdown (cf. \cite{WORS3}).
\end{remark}

\subsection{Preliminaries on adjoint groups of simple Lie algebras}\label{ssec:notgrps}

The main characters in this paper will be the equalized $\C^*$-actions on rational homogeneous varieties (\emph{RH varieties}, for short) of Picard number one. Here, we will introduce some notations we will use regarding these varieties and the corresponding $\C^*$-actions.

Throughout this paper, $G$ will denote the adjoint group of a simple Lie algebra $\fg$. 
We consider a Cartan subgroup $H\subset G$, a Borel subgroup $H\subset B\subset G$ and their corresponding Lie algebras $\fh\subset \fb\subset \fg$. 
We denote by $\Phi$ the root system of $G$ with respect to $H$, by $\Delta=\{\alpha_1,\dots,\alpha_n\}\subset \Phi^+\subset \Phi$ a set of positive simple roots of $\fg$ (determined by the choice of $B$), by $\omega_1,\dots,\omega_n$ the corresponding fundamental weights, by $W=\No_G(H)/H$ the Weyl group of $G$, and by  $s_1,\dots,s_n\in W$ the reflections associated to these positive simple roots. 
We will denote by $\cD$ the Dynkin diagram of $\fg$, whose nodes are numbered by $D=\{1,\dots,n\}$ (we follow the numbering convention of \cite[Planche I-IX]{Bourb}), in one-to-one correspondence with the set of positive simple roots $\Delta$. 
We will denote by $w_0\in \No_G(H)\subset G$ an element in the class of the longest element of the Weyl group $W= \No_G(H)/H$ (with respect to the choice of $\Delta\in \Phi$). 

Each $G$-homogeneous variety of Picard number one is completely determined by the choice of a node $k$ in $\cD$, since this determines uniquely a fundamental weight $\omega_k$ of the simple Lie algebra $\fg$ associated with $\cD$.
The closed $G$-orbit in the projectivization $\P\left(V(\omega_k)\right)$ of the irreducible representation $V(\omega_k)$ associated with $\omega_k$ is an RH $G$-variety of Picard number one, that we denote by $\cD(k)$. 
Alternatively, given the node $k$ of the Dynkin diagram $\cD$ and the subgroup $W_k\subset W$  generated by the reflections $s_i$, $i\neq k$, the subgroup $P:=BW_kB\subset G$ is a parabolic subgroup containing $B$, and 
\[
G/P\simeq \cD(k).
\]
We stress out that, the subgroup $P$ is completely determined by its Lie algebra $\fp\subset\fg$, whose Cartan decomposition is:
\begin{equation}\label{eq:Cartanparab}
\fp=\underbrace{\bigoplus_{\beta\in\Phi^+}\fg_\beta\oplus\fh}_{\fb}\oplus \bigoplus_{\substack{\beta\in\Phi^+:\\\sigma_k(\beta)=0}}\fg_{-\beta}.
\end{equation}
Here $\sigma_k:\Mo(H)\to \Z$ denotes the \emph{height map} determined by sending $\alpha_j$ to $1$ if $j=k$, and to $0$ otherwise, where $\Mo(H)$ is the group of characters of $H$, cf. \cite{Tev05}.

\begin{notation}\label{not:outNodes}
	Consider the Dynkin diagram $\cD$, whose nodes are numbered by elements of $D=\{1,\ldots,n\}$. We define the RH varieties $\cD(0)$ and $\cD(n+1)$ to be isolated points.
\end{notation}

\subsection{Preliminaries on RH varieties of classical type}\label{ssec:notclassic}

\subsubsection{Projective description}\label{sssec:projDescr}

For the reader's convenience, we include here the standard projective descriptions of the RH varieties of Picard number one of classical type. Moreover, we introduce some notations about them that will be useful later on.

\begin{itemize}[leftmargin=*]
\item For every $k\in \{1,\dots,n\}$ the variety $\DA_n(k)$ is the \emph{Grassmannian} of $k$-dimensional subspaces of the $(n+1)$-dimensional vector space $V$, which is the standard representation of the Lie algebra of type $\DA_n$. 
\item For $k\in \{1,\dots,n\}$ the variety $\DB_n(k)$ is the \emph{orthogonal Grassmannian} parametrizing $k$-dimensional vector subspaces of a $(2n+1)$-dimensional vector space $V$ that are isotropic with respect to a fixed bilinear symmetric form of maximal rank. In particular, $\DB_n(1)$ is a $(2n-1)$-dimensional smooth quadric.
\item For $k\in \{1,\dots,n\}$, $\DC_n(k)$ is the \emph{symplectic Grassmannian}, which parametrizes $k$-dimensional vector subspaces of a $2n$-dimensional vector space $V$ that are isotropic with respect to a fixed bilinear skew-symmetric form of maximal rank. In particular, $\DC_n(1)\simeq \P(V)$. 
\item For the description of the varieties $\DD_n(k)$, called \emph{orthogonal Grassmannians} as well, we consider a $2n$-dimensional vector space $V$ equipped with a bilinear symmetric form of maximal rank. Then, for $k\leq n-2$, $\DD_n(k)$ parametrizes $k$-dimensional vector subspaces of $V$ that are isotropic with respect to such a symmetric form. In particular $\DD_n(1)$ is a $(2n-2)$-dimensional smooth quadric. The varieties $\DD_n(n-1),\DD_n(n)$ parametrize the two families of $n$-dimensional isotropic subspaces of $V$. Any $(n-1)$-dimensional isotropic subspace of $V$ is the intersection of precisely two isotropic subspaces of dimension $n$, one in each family $\DD_n(n-1),\DD_n(n)$. The orthogonal Grassmannian of $(n-1)$-dimensional isotropic subspaces of $V$, denoted by $\DD_n(n-1,n)$ has Picard number two (and so it will not be considered in this paper), with two contractions:
	\[
	\xymatrix{ & \DD_n(n-1,n) \ar[ld]_{\P^{n-1}} \ar[rd]^{\P^{n-1}} & \\ \DD_n(n-1) & & \DD_n(n).}
	\]
\end{itemize}

\subsubsection{Universal bundles}\label{sssec:univBundles}
Grassmannians (standard, orthogonal and symplectic) come equipped with some universal bundles that we will now describe. 

\begin{itemize}[leftmargin=*]
\item The Grassmannian $\DA_n(k)$ supports two universal vector bundles $\cS,\cQ$ of rank $k$ and $n+1-k$,  respectively. The former is the subbundle of the trivial bundle, whose fiber over an element $[W]\in \DA_n(k)$ is the corresponding subspace $W\subset V$. Then $\cQ$ is defined as the cokernel of the inclusion $\cS\hookrightarrow V\otimes \cO_{\DA_n(k)}$, so that we have the short exact sequence:
\[
\shse{\cS}{V\otimes\cO_{\DA_n(k)}}{\cQ}.
\]
It is then well known that we have a decomposition of the tangent bundle of $\DA_n(k)$ as
\[
T_{\DA_n(k)}\simeq \cS^\vee \otimes \cQ.
\]
\item In the case of \emph{isotropic Grassmannians} (i.e. orthogonal or symplectic), we may consider the pullback of the two universal bundles on standards Grassmannians via the natural inclusions $\DB_n(k)\subset \DA_{2n+1}(k)$, $\DC_n(k),\DD_n(k)\subset \DA_{2n}(k)$ (for $k\le n-2$ in the $\DD_n$-case), $\DD_{n}(n-1),\DD_n(n)\subset  \DA_{2n}(n)$. In order to consider all these cases together, we start with a finite-dimensional vector space $V$ endowed with a nondegenerate symmetric or skew-symmetric isomorphism $q:V\to V^\vee$. For every $k<\dim V/2$, consider the variety $X\subset \DA_{\dim V-1}(k)$ parametrizing isotropic $k$-dimensional subspaces of $V$ and the restrictions of the universal bundles -- that we still denote by $\cS$ and $\cQ$ -- on these isotropic Grassmannian.  
The fact that the subspaces parametrized by $X$ are $q$-isotropic  implies that the composition
\[
\cS \lra V \otimes \cO_X\stackrel{q}{\lra} V^\vee \otimes \cO_X \lra  \cS^\vee
\]
is zero, and so we have an injection $\cS\to \cQ^\vee$, whose cokernel we denote by $\cK$.
The isomorphism $q$ induces a (symmetric or skew-symmetric, respectively) isomorphism $\cK\simeq \cK^\vee$ and so, summing up, we have a commutative diagram with short exact rows and columns:
\begin{equation}
\xymatrix@R=10mm{\cS \ar@{=}[d] \ar[r] & \cQ^\vee \ar[d] \ar[r] & \cK \simeq \cK^\vee \ar[d] \\ 
\cS \ar[r] \ar[d] & V \otimes \cO \stackrel{q}{\simeq}  V^\vee \otimes \cO \ar[d] \ar[r] &\cQ \ar[d]\\
0 \ar[r] & \cS^\vee \ar@{=}[r] & \cS^\vee}
\label{eq:univBD}
\end{equation}  
Furthermore, it is known for example by \cite[Proposition 5.1 and  5.4]{LM} that the tangent bundle of $X$ fits into a  short exact sequence:
\begin{equation}\label{eq:tangentBD}
\shse{\cS^\vee \otimes \cK}{T_{X}}{\cC}
\end{equation}
where
\[
\cC=\begin{cases}\bigwedge^2\cS^\vee &\text{if }\cD=\DB_{(\dim V-1)/2},\DD_{\dim V/2},\\
S^2\cS^\vee &\text{if }\cD=\DC_{\dim V/2}.\end{cases}
\]
Note that in the case where $X$ parametrize maximal isotropic subspaces, i.e. $X=\DC_n(n),\DD_n(n-1),\DD_n(n)$, the bundle $\cK$ is equal to zero, so with the above notation we will have $T_X\simeq \cC$ (cf. \cite[Section 3.1]{LM}). 
\end{itemize} 

\subsection{Equalized $\C^*$-actions on RH varieties}\label{ssec:equalized}

Let us now describe briefly equalized $\C^*$-actions on rational homogeneous varieties of Picard number one as above. We will use the notation introduced in Section \ref{ssec:basic}.

\subsubsection{Short gradings}\label{sssec:shortgr}
Following \cite{Fra1,WORS5}, equalized $\C^*$-actions on RH varieties are given by the choice of a {\em short grading} on $\fg$:
\[
\fg=\fg_{-}\oplus \fg_0\oplus \fg_+.
\]
Up to the adjoint action of an element of $W$, we may assume that $\fb\subset \fg_0\oplus \fg_+$, and  the grading will then be completely determined by the choice of an index $i\in D=\{1,\dots,n\}$ such that the corresponding height map $\sigma_i:\Mo(H)\to \Z$ satisfies that $\sigma_i(\Phi)=\{-1,0,1\}$. 
We will denote
\[
\fp_{\pm}:= \fg_0\oplus \fg_{\pm},
\]
which are parabolic subalgebras of $\fg$, corresponding to opposite parabolic subgroups $P_\pm\subset G$, cf. \cite[Section~4.8]{BT}.  

\begin{remark}\label{rem:definitionHi}
We have $B\subset P_+=BW_iB$, so that, with the marked Dynkin diagram notation, $G/P_+\simeq \cD(i)$ and $P_\pm$ have a common Levi part $G_0$ with Lie algebra $\fg_0$.

The reductive subalgebra $\fg_0\subset \fg$ decomposes as $\fg^\perp\oplus \fh_i$, where $\fh_i$ is a $1$-dimensional vector subspace of the Cartan subalgebra $\fh$, and $\fg^\perp$ is semisimple. 
The subalgebras $\fg^\perp,\fg_0,\subset \fg$ and $\fh_i \subset \fh$ correspond to  subgroups $G^\perp\subset G_0\subset G$ and $H_i\subset H$.
The Dynkin diagram $\cD^\perp$ associated with $G^\perp$ can be obtained by deleting the node $i$ from the Dynkin diagram $\cD$. 

Moreover, $H_i$ can be described as the image of the induced injective map $\sigma_i^*:\C^*\to H$. In the sequel we will directly identify $H_i$ with $\C^*$ via $\sigma_i^*$. 
\end{remark}
Then the following statement follows from \cite{Fra1}:

\begin{proposition}\label{prop:Hi}
With the above notation, the $H_i$-action on each rational homogeneous $G$-variety $X=G/P$ is equalized, if and only if the short grading induced by $\sigma_i$ on $\fg$ is short. 
\end{proposition}

\begin{remark}\label{rem:projectionM(H)}
The inclusion $\imath:G^\perp\hookrightarrow G$, together with the choice of the Cartan subgroup $H^\perp:=G^\perp \cap H$, define a group homomorphism $\imath^*:\Mo(H)\to \Mo(H^\perp)$ sending $\{\alpha_j : j\neq i\}$ to the set of positive simple roots of $G^\perp$ corresponding to the choice of the Borel subgroup $B^\perp:=G^\perp \cap B$ and sending the fundamental weight $\omega_i$ to zero. 
We denote by $\alpha_j\in \Mo(H^\perp)$ the image of $\alpha_j\in \Delta$ via $\imath^*$ for every $j\neq i$. 
\end{remark}

The following table describes the list of possible height maps that induce a short grading on $\fg$:

\begin{table}[ht!]
\begin{tabular}{|c||c|c|c|c|c|c|}
\hline
	$\fg$ & $\DA_n$ & $\DB_n$ & $\DC_n$ & $\DD_n$ & $\DE_6$ & $\DE_7$ \\
	\hline
	$\sigma_i$ & $\sigma_i$ for $i=1,\ldots,n$ & $\sigma_1$ & $\sigma_n$ & $\sigma_1,\sigma_{n-1},\sigma_n$ & $\sigma_1,\sigma_6$ & $\sigma_7$\\\hline
\end{tabular}
\caption{Short gradings of simple Lie algebras. \label{tab:short}}
\end{table}

For the RH varieties of classical type, equalized actions have explicit descriptions, that we will present in detail in Section \ref{sec:equalGrass}.

\subsubsection{Fixed-point components}\label{sssec:fpc}
We will now describe the fixed-point component of an equalized $H_i$-action as in Section \ref{sssec:shortgr} on an RH variety $G/P \simeq \cD(k)$ as in Section \ref{ssec:notgrps}, with $i,k\in D$. 
For every $w\in W$, we will denote 
\[
P^\perp_w:=G^\perp\cap \conj_w(P)
\]
which is a parabolic subgroup of $G^\perp$. 
In particular, $P^\perp:=P^\perp_e$, where $e \in W$ is the unity. 
Such a parabolic subgroup is completely determined by its Lie algebra:
\begin{equation}\label{eq:parabW}
\fp^\perp_w=\left(\bigoplus_{\substack{\beta\in\Phi^+:\\\sigma_i(w(\beta))=0}}\fg_{w(\beta)}\right)\oplus \left(\Ad_w(\fh)\cap\fh^\perp\right)\oplus \left(\bigoplus_{\substack{\beta\in\Phi^+:\\\sigma_k(\beta)=0\\\sigma_i(w(\beta))=0}}\fg_{-w(\beta)}\right) \subset \fg^\perp.
\end{equation}

Following \cite[Corollary~3.10]{WORS5}, the fixed-point components of the $H_i$-action are $G^\perp$-homogeneous varieties. 
Since each fixed-point component $Y \subset G/P$ must contain a fixed-point for the $H$-action (which are the points of the form $wP$ for $w\in W$, see \cite[Section~3.4]{CARRELL}), $Y$ can then be written as:
\[
Y={G^\perp}/{P^\perp_w}.
\]
As a rational homogeneous $G^\perp$-variety, $Y \simeq \cD^\perp(J)$, for some $J\subset D\setminus \{i\}$.

The Picard group of $G/P \simeq \cD(k)$ is generated by the homogeneous line bundle $L$ determined by the fundamental weight $\omega_k$. 
Then, following \cite[Corollary 3.8]{WORS5}, the $\C^*$-action admits a linearization on $L$ such that the $L$-weight of the fixed-point component passing by $wP$ is 
\[
\mu_L(G^\perp/P^\perp_w)=\sigma_i(\omega_k-w(\omega_k)).
\]
The minimum and maximum value of $\mu_L$ are achieved at $w=e$ and $w=w_0$.
So we may conclude that the sink and the source of the $H_i$-action are, respectively:
\begin{equation}\label{eq:sinkSource}
	Y_-:={G^\perp}/{P^\perp},\qquad Y_+:= {G^\perp}/{P^\perp_{w_0}}.
\end{equation}

\begin{remark}\label{rem:Y-Y+marked}
In particular, with the marked Dynkin diagram notation, if $i \ne k$,
\[
Y_-\simeq\cD^\perp(k),\qquad Y_+ \simeq \cD^\perp(w_0(k)),
\]
where $w_0(k)$ denotes the node of the Dynkin diagram $\cD^\perp \subset \cD$ corresponding to the positive simple root $-w_0(\alpha_k) \in \Delta$.
\end{remark}

\subsubsection{The normal bundle of the extremal fixed-point components}\label{sssec:normalbundle}
We end up this section by describing the normal bundles of the sink and source $Y_\pm$ of the $H_i$-action into the rational homogeneous variety $X=G/P\simeq\cD(k)$, which, by the Bia{\l}ynicki-Birula theorem, are isomorphic to two $\C^*$-invariant open sets of $X$. 

\begin{lemma}\label{lem:normals}
With the above notation, the normal bundles of $Y_\pm\subset X$ of the $H_i$-action on $X$ are the homogeneous $G^\perp$-bundles
\[
\cN_{Y_-|X}\simeq G^\perp\times^{P^{\perp}} N_-, \qquad \cN_{Y_+|X}\simeq G^\perp\times^{P^{\perp}_{w_0}} N_+.
\]
where $N_-,N_+$ are, respectively, the $P^\perp, P^\perp_{w_0}$-submodules of $\fg$ defined by:
\begin{equation}\label{eq:Ns}
N_-:=\bigoplus_{\substack{\beta\in\Phi^+\\ \sigma_i(\beta)>0\\\sigma_k(\beta)>0}} \fg_{-\beta}\subset\fg_-, \qquad
N_+:=\bigoplus_{\substack{\beta\in\Phi^+\\ \sigma_i(\beta)>0\\\sigma_k(w_0(\beta))<0}} \fg_{\beta}\subset\fg_+.
\end{equation}
Moreover the $H^\perp$-weight of every subspace $\fg_{\beta}$ is equal to $\imath^*\beta$.
\end{lemma}

\begin{proof}
We will do the proof in the case of $Y_-$, being $Y_+$ analogous. Note that the action of $G_0$ on $Y_-$ lifts to $\cN_{Y_-|X}$, so we may write
\[
\cN_{Y_-|X}=G_0\times^{P_0}N_-.
\]
where $N_-=\cN_{Y_-|X,eP}$ and $P_0:=G_0\cap P$. 
In order to compute $N_-$ as a $P_0$-module, we start from the Cartan decomposition of $\fg$ and note that:
\begin{align*}
\fg &=\left(\bigoplus_{\beta\in\Phi^+}\fg_\beta\right)\oplus\fh\oplus\left(\bigoplus_{\beta\in \Phi^+}\fg_{-\beta}\right),\\[3pt]
\fp &=\left(\bigoplus_{\beta\in\Phi^+}\fg_\beta\right)\oplus\fh\oplus\left( \bigoplus_{\substack{\beta\in\Phi^+:\\\sigma_k(\beta)=0}}\fg_{-\beta}\right),\\
\fg_0&=\left(\bigoplus_{\substack{\beta\in\Phi^+:\\\sigma_i(\beta)=0}}\fg_\beta\right)\oplus\fh\oplus\left( \bigoplus_{\substack{\beta\in\Phi^+:\\\sigma_i(\beta)=0}}\fg_{-\beta}\right),\\[3pt]
\fg_0 \cap \fp&=\left(\bigoplus_{\substack{\beta\in\Phi^+:\\\sigma_i(\beta)=0}}\fg_\beta\right)\oplus\fh\oplus\left( \bigoplus_{\substack{\beta\in\Phi^+:\\\sigma_k(\beta)=0\\\sigma_i(\beta)=0}}\fg_{-\beta}\right).
\end{align*}
Then we can conclude that $N_-$ is isomorphic as a $P_0$-module to the quotient $(\fg/\fp)/(\fg_0/\fg_0\cap \fp)$, hence to
\[
\bigoplus_{\substack{\beta\in\Phi^+:\\\sigma_i(\beta)>0\\\sigma_k(\beta)>0}} \fg_{-\beta}.
\]
To reduce the action on $N_-$ from $P_0$ to $P^\perp$ and compute the $H^\perp$-weights, we apply the group homomorphism $\imath^*:\Mo(H)\to \Mo(H^\perp)$, see Remark \ref{rem:projectionM(H)}. 
\end{proof}



\section{Equalized actions on 
Grassmannians}\label{sec:equalGrass} 


For every case of Table \ref{tab:short} of classical type, we describe in this section the fixed-point component of the $\C^*$-action $H_i$ induced by the height map $\sigma_i$ on each RH $G$-variety $X$ of Picard number one.
We denote by $L$ the generator of the Picard group.
We proceed as in \cite[Section 4]{Fra1} (see also \cite[Example 7.12]{WORS1}): we first describe the $H_i$-action on the standard representation $V$ and on its closed $G$-orbit in $\P(V)$, then we study the induced $H_i$-action on the other $G$-homogeneous varieties, see Section \ref{sssec:projDescr} for the notation.

\subsection{$\DA_n$-varieties}\label{ssec:An}

We consider the $H_i$-action on $\DA_n(k)$, the Grassmannian of $k$-dimensional subspaces of the $(n+1)$-dimensional vector space $V$, corresponding to the standard representation $V(\omega_1)$. For simplicity we will assume that $k\leq n-k+1$. 
The $H_i$-action on $\DA_n(1)\simeq \P(V)$ is induced by the linear action on $V$ given by
	\[
	t\cdot (x_0,\ldots,x_n) = (tx_0,\ldots,tx_{i-1},x_i, \ldots,x_n),
	\]
on a certain set of coordinates. 
In particular, we get a decomposition $V=V_-\oplus V_+$, where 
\[
V_-=\left\{x_i=\ldots=x_n=0\right\}, \qquad V_+=\left\{x_0=\ldots=x_{i-1}=0\right\}.
\]
Then the $H_i$-action on $\DA_n(k)$ is obtained as the induced action on the Grassmannians of $k$-subspaces of $V$.
 In every case, the only fixed-point component of weight $s$, for every $s$, is:
\[
\DA_{i-1}(k-s)\times \DA_{n-i}(s),
\] 
respectively, which is the product of two Grassmannians of vector subspaces of $V_-$ and $V_+$. 
By Notation \ref{not:outNodes}, we recall that the factor $\DA_{i-1}(k-s)$ (resp. $\DA_{n-i}(s)$) is a point if $s=k,k-i$ (resp. $s=0,n-i+1$) . Hence 
we have three possible cases according to the values of $i$, $k$ and $n$: 

\addtolength{\tabcolsep}{-4.5pt}
\begin{table}[h!]
\begin{tabular}{|c|c|c|c|c|}
\hline
&$s$
&$Y_-$&$Y_s$&$Y_+$\\\hline\hline
$i\leq k$
&$[k-i,k]$&$\DA_{n-i}(k-i)$&$\DA_{n-i}(s)$&$\DA_{n-i}(k)$\\\hline
$i\!\in\![k,n-k+1]$
&$[0,k]$&$\DA_{i-1}(k)$&$\DA_{i-1}(k-s)\!\times\! \DA_{n-i}(s)$
&$\DA_{n-i}(k)$\\\hline
$i\geq n-k+1$
&$[0,n-i+1]$&$\DA_{i-1}(k)$&$\DA_{i-1}(k-s)$&$\DA_{i-1}\!(i-n+k-1)$\\\hline
\end{tabular}
\end{table}
\addtolength{\tabcolsep}{+2pt}
Note that $Y_-$ is given by the point $[V_-]$ if and only if $i=k$. On the other hand $Y_+$ is the point $[V_+]$ if and only if $i=n-k+1$. In the rest of the cases, the extremal fixed-point components of the action are positive dimensional Grassmannians.

\begin{remark}
 Composing the $H_i$-action with the map $t \mapsto t^{-1}$ and with the permutation $(x_0,\dots,x_n)\mapsto (x_n,\dots,x_0)$, we obtain the $H_{n-i}$-action. In other words, we may assume without loss of generality that $i\leq n-i+1$, and in particular we may discard the third case in the above table.
 \end{remark} 

\subsection{$\DB_n$-varieties}\label{ssec:Bn}

We consider the $H_1$-action on $\DB_n(k)$, the orthogonal Grassmannian of $k$-dimensional subspaces of the $(2n+1)$-dimensional standard representation $V$, isotropic with respect to a given nondegenerate symmetric form. We may choose coordinates $(x_0,\dots, x_{2n})$ (with respect to a certain basis $\{e_0,\dots,e_{2n}\}$) on $V$ so that this form is represented by a symmetric (block) matrix:
\[
\begin{pmatrix}[c|c|c]
0&I_n&0\\\hline
I_n&0&0\\\hline
0&0&1
\end{pmatrix}
\]
where $I_n$ denote the identity $n\times n$ matrix.
 In particular, $\DB_n(1)\subset \P(V)$ is the quadric hypersurface $Q^{2n-1}$ given by the equation 
 \[
 x_0x_{n}+\ldots+x_{n-1}x_{2n-1}+x_{2n}^2=0.
 \]
 In this case, the $H_1$-action on every $\DB_n(k)$ is induced by the linear action on $V$:
\begin{equation}\label{equation:H1-actionQuadric}
	t \cdot (x_0,\ldots,x_{2n})=(tx_0,x_1,\ldots,x_{n-1},t^{-1}x_n,x_{n+1},\ldots,x_{2n-1},x_{2n}),
\end{equation}
giving a decomposition $V=V_-\oplus V_0\oplus V_+$, where we have
\begin{align*}
	&V_-=\left\{x_1=\ldots=x_{2n}=0\right\},\\
	&V_0=\left\{x_0=x_n=0\right\},\\
	&V_+=\left\{x_0=\ldots=x_{n-1}=x_{n+1}=\ldots=x_{2n}=0\right\}.
\end{align*}
The vector subspace $V_0\subset V$ inherits a nondegenerate symmetric bilinear form, whose associated Lie algebra (which is of type $\DB_{n-1}$) is precisely $\fg^\perp$. 
We have two cases:

\begin{table}[h!]
\begin{tabular}{|c|c|c|c|c|}
\hline
&$[\mu_L(Y_-),\mu_L(Y_+)]$&$Y_-$&$Y_{(\mu_L(Y_+)+\mu_L(Y_-))/2}$&$Y_+$\\\hline\hline
$k<n$&$[0,2]$&$\DB_{n-1}(k-1)$&$\DB_{n-1}(k)$&$\DB_{n-1}(k-1)$\\\hline 
$k=n$&$[0,1]$&$\DB_{n-1}(n-1)$&$\emptyset$&$\DB_{n-1}(n-1)$\\\hline
\end{tabular}
\end{table}

The sink and the source of the action parametrize isotropic subspaces of $V$ containing, respectively, $V_-$ and $V_+$; every such subspace intersects $V_0$ in a $(k-1)$-dimensional isotropic subspace of $V_0$. When $k<n$ we also have an inner fixed-point component, parametrizing $k$-dimensional isotropic subspaces of $V_0$.

We note here that the value of $\mu_L(Y_+)$ changes because the ample generator $L$ of $\Pic(\DB_n(k))$ is the Pl\"ucker line bundle for $k<n$, and its square root for $k=n$. 

\subsection{$\DC_n$-varieties}\label{ssec:Cn}

We consider the $H_n$-action on $\DC_n(k)$, the symplectic Grassmannian of $k$-dimensional subspaces of the $2n$-dimensional standard representation $V$, isotropic with respect to a given nondegenerate skew-symmetric form; we may assume this form to be given, in a certain set of coordinates, by the matrix 
\[
\begin{pmatrix}[c|c]
0&I_n\\\hline
-I_n&0
\end{pmatrix}.
\]
The $H_n$-action on every $\DC_n(k)$ is induced from the linear action on $V$ given by: 
\begin{equation}\label{equation:Hn-actionProjectiveSpace}
	t \cdot (x_0,\ldots,x_{2n-1})=(tx_0,\ldots,tx_{n-1},x_n,\ldots,x_{2n-1}).
\end{equation}
In other words, it is the restriction to $\DC_n(k)$ of the $H_n$-action on the Grassmannian $\DA_{2n-1}(k)$. The fixed-point components of $\DC_n(k)$ can be computed by intersecting the fixed-point components of $\DA_{2n-1}(k)$ with this variety. Note also that the skew-symmetric form induces an isomorphism $q: V_- \to V_+^\vee$, that extends to isomorphisms between Grassmannians of subspaces of $V_-$ and subspaces of $V_+$ of complementary dimension. 
With this information at hand one may compute fixed-point components of the $H_n$-action on $\DC_n(k)$:
\begin{table}[h!]
\begin{tabular}{|c|c|c|c|c|}
\hline
&$[\mu_L(Y_-),\mu_L(Y_+)]$&$Y_-$&$Y_s$&$Y_+$\\\hline\hline
$k< n$&$[0,k]$&$\DA_{n-1}(k)$&$\DA_{n-1}(k-s,n-s)$&$\DA_{n-1}(n-k)$\\\hline
$k=n$&$[0,n]$&$\point$&$\DA_{n-1}(n-s)$&$\point$\\\hline
\end{tabular}
\end{table}

We note that, in the case $k=n$, every inner fixed-point component $Y_s\simeq\DA_{n-1}(n-s)$ is the image of the map
\[
	\DA_{n-1}(n-s)\lra\DA_{n-1}^{}(n-s)\times\DA_{n-1}^{}(s)\subset\DA_{2n-1}(n),\qquad	W\longmapsto \left(W,\, \ker q(W)\right),
\]
which in particular tells us that the restriction to $Y_s\simeq\DA_{n-1}(n-s)$ of the Pl\"ucker embedding of $\DC_n(n)$ is not the minimal embedding but its second symmetric power.

\subsection{$\DD_n$-varieties}\label{ssec:Dn}

We consider now equalized actions on the orthogonal Grassmannian of $k$-dimensional subspaces of the $2n$-dimensional standard representation $V$, isotropic with respect to a given nondegenerate symmetric form. We choose coordinates $(x_0,\dots,x_{2n-1})$ in $V$ (with respect to a basis $\{e_0,\dots,e_{2n-1}\}$) so that this form is given by the matrix:
\[
\begin{pmatrix}[c|c]
0&I_n\\\hline
I_n&0
\end{pmatrix}.
\] 

We have three equalized $\C^*$-actions, coming from the action on $V$ that take the form:
\begin{align*}
	H_1:& & &t \cdot (x_0,\ldots,x_{2n-1})=(tx_0,x_1,\ldots,x_{n-1},t^{-1}x_n,x_{n+1},\ldots,x_{2n-1}),\\
	H_{n-1}:& & & t \cdot (x_0,\ldots,x_{2n-1})=(tx_0,\ldots,tx_{n-2},x_{n-1},x_{n},\ldots,x_{2n-2},tx_{2n-1}),\\
	H_n:& & & t \cdot (x_0,\ldots,x_{2n-1})=(tx_0,\ldots,tx_{n-2},tx_{n-1},x_{n},\ldots,x_{2n-2},x_{2n-1}).
\end{align*}

\begin{remark}
Up to reordering the coordinates $x_{n-1},x_{2n-1}$ we have only two types of equalized $\C^*$-actions, $H_1$ and $H_n$. 
\end{remark}

\subsubsection{$H_1$-action}\label{sssec:H1Dn}
The case of the $H_1$-action on the varieties $\DD_n(k)$ is analogous to the case $\DB_n$, Section \ref{ssec:Bn}, and the fixed-point components are Grassmannians of isotropic subspaces in $V_0$, with respect to the restriction of the symmetric form in $V$. We get here four cases:  
\begin{table}[ht!]
\begin{tabular}{|c|c|c|c|c|}
\hline
&$[\mu_L(Y_-),\mu_L(Y_+)]$&$Y_-$&$Y_{(\mu_L(Y_+)+\mu_L(Y_-))/2}$&$Y_+$\\\hline\hline
$k<n-2$&$[0,2]$&$\DD_{n-1}(k-1)$&$\DD_{n-1}(k)$&$\DD_{n-1}(k-1)$\\\hline
$k=n-2$&$[0,2]$&$\DD_{n-1}(n-3)$&$\DD_{n-1}(n-2,n-1)$&$\DD_{n-1}(n-3)$\\\hline
$k= n-1 $&$[0,1]$&$\DD_{n-1}(n-2)$&$\emptyset$&$\DD_{n-1}(n-1)$\\\hline
$k= n $&$[0,1]$&$\DD_{n-1}(n-1)$&$\emptyset$&$\DD_{n-1}(n-2)$\\\hline
\end{tabular}
\end{table}

Note that, as in the $\DB_n$-case, the value of $\mu_L(Y_+)$ changes because the ample generator $L$ of $\Pic(\DD_n(k))$ is the Pl\"ucker line bundle for $k<n-1$, and its square root for $k= n-1$ or $k=n$.

\subsubsection{$H_n$-action}\label{sssec:HnDn}
The case is analogous to the $H_n$-action on the varieties of type $\DC_n$. In fact, the linear action on $V$ is the same as \eqref{equation:Hn-actionProjectiveSpace}.
We will denote again
\[
V_-=\left\{x_n=\ldots=x_{2n-1}=0\right\},\qquad V_+=\left\{x_0=\ldots=x_{n-1}=0\right\}.
\]
They are isotropic subspaces of maximal dimension and, up to renumbering the nodes of $\DD_n$ we may assume that $V_-\in \DD_n(n)$. 
We will then distinguish two cases:
\begin{itemize}
\item[(A)] $n$ is even and $V_+\in\DD_n(n)$. 
\item[(B)] $n$ is odd and $V_+\in \DD_n(n-1)$.
\end{itemize}
Hence for the $H_n$-action on $\DD_n(k)$ we have five cases: 

\begin{table}[h!]
\begin{tabular}{|c|c|c|c|c|}
\hline
&$[\mu_L(Y_-),\mu_L(Y_+)]$&$Y_-$&$Y_s$&$Y_+$\\\hline\hline
$k<n-1$&$[0,k]$&$\DA_{n-1}(k)$&$\DA_{n-1}(k-s,n-s)$&$\DA_{n-1}(n-k)$\\\hline
$k= n-1$ (A)&$[1, n/2]$&$\DA_{n-1}(n-1)$&$\DA_{n-1}(n-1-2s)$&$\DA_{n-1}(1)$\\\hline
$k= n$ (A)&$[0,n/2]$&$\point$&$\DA_{n-1}(2s)$&$\point$\\\hline
$k= n-1$ (B)&$[0,(n-1)/2]$&$\DA_{n-1}(n-1)$&$\DA_{n-1}(n-1-2s)$&$\point$\\\hline
$k= n$ (B)&$[0,(n-1)/2]$&$\point$&$\DA_{n-1}(2s)$&$\DA_{n-1}(1)$\\\hline
\end{tabular}
\end{table}



\section{$\C^*$-actions and inversion in Jordan algebras: balanced gradings}\label{sec:j-str}


This section is devoted to study equalized $\C^*$-actions on rational homogeneous varieties whose associated birational map is a Cremona transformation of a projective space, and we will prove that they can be understood in the language of Jordan algebras.

Consider a 1-dimensional torus $H_i\subset H \subset G$, determined by a height map $\sigma_i$ producing a short grading (see Section \ref{sssec:shortgr}):
\[
\fg=\fg_{-}\oplus\fg_0\oplus\fg_+.
\]
Again, let $w_0\in \No_G(H)\subset G$ be an element in the class of the longest element of the Weyl group $W$ of $G$ (see Section \ref{ssec:notgrps}). 

\begin{definition}\label{def:balanced}
The short grading of $\fg$ determined by $\sigma_i:\Mo(H)\to \Z$ is said to be \emph{balanced}  if $t^{-1}=\conj_{w_0}(t)$ for every $t\in H_i$. In this case we will say that the action of $H_i$ on every RH variety $G/P$ is \emph{balanced}.
\end{definition}

Note that this property does not depend on the choice of the particular representative $w_0\in \No_G(H)$. 
In particular, denoting by $L_g$ the left multiplication by $g$ for every $g\in G$, we have the following trivial statement:

\begin{lemma}\label{lem:balanced}
With the notations as above, if the short grading of $\fg$ determined by $\sigma_i:\Mo(H)\to \Z$ is balanced, then for every parabolic subgroup $P\subset G$ and for every $t \in H_i$ the following diagram is commutative:
\[
\xymatrix{G/P\ar[d]_{L_{w_0}}\ar[r]^{L_{t^{-1}}}&G/P\ar[d]^{L_{w_0}}\\
G/P\ar[r]_{L_t}&G/P}
\]
\end{lemma}

Then we may prove the following:

\begin{lemma}\label{lem:balshort}
	The complete list of balanced short gradings on simple Lie algebras is given in the following table:
 \begin{table}[ht!]
\begin{tabular}{|c||c|c|c|c|c|c|}
\hline
	$\fg$ & $\DA_{2n-1}$ & $\DB_n$ & $\DC_n$ & $\DD_n$ & $\DD_{2n}$  & $\DE_7$ \\
	\hline
	$\sigma_i$ & $\sigma_n$ & $\sigma_1$ & $\sigma_n$ & $\sigma_1$ & $\sigma_{2n},\sigma_{2n-1}$ & $\sigma_7$\\\hline
\end{tabular}
\caption{Balanced short gradings of simple Lie algebras. \label{tab:balshort}}
\end{table}		
\end{lemma}

\begin{proof}
We will use that, being $G$ the adjoint group of $\fg$, the adjoint representation of $G$ is faithful. Then the balanced condition is equivalent to saying that $$\Ad_{w_0}\circ\Ad_t\circ\Ad_{w_0^{-1}}=\Ad_{t^{-1}}\quad\mbox{for every $t\in H_i$,}$$ as endomorphisms of $\fg$. Since they trivially coincide on the Cartan subalgebra $\fh$, the two maps above are equal if they coincide on $\fg_\beta$ for every $\beta\in \Phi$. Then the balanced condition is equivalent to saying that $H_i$ is contained in the following subgroup of $H$:
\begin{equation}\label{eq:balanced}
H':=\left\{t\in H|\,\, t^{w_0(\beta)}=t^{-\beta},\,\,\mbox{for every }\beta\in \Phi\right\}.
\end{equation}
Here we are denoting by $w_0$ the automorphism of $\Mo(H)$ induced by $w_0$, that is the class of $w_0$ in the Weyl group $W$. It is a well-known fact (see for example \cite{Bourb}) that $w_0$ coincides with $-\id$ in the cases in which the Dynkin diagram $\cD$ of $\fg$ has no non-trivial automorphisms and in the case $\DD_{2n}$. We conclude that in these cases $H'=H$, and so every short grading is balanced.

We are then left with the cases $\DA_n$, $\DD_{2n-1}$ and $\DE_6$, in which $w_0$ is equal to the composition of $-\id$ with the homomorphism induced by the permutation of the positive simple roots $\Delta$ given by the only nontrivial automorphism $s$ of $\cD$. Then a straightforward computation in each case shows that the complete list of balanced short gradings is the one given in Table \ref{tab:balshort}.
\end{proof}

In particular, consider the parabolic subgroups $P_{\pm}\subset G$ defined by $H_i$ and the corresponding short grading, we obtain the following consequence:

\begin{corollary}\label{cor:balshort}
With the notation above, consider the $H_i$-action on a rational homogeneous $G$-variety $G/P$. Then the following conditions are equivalent:
\begin{itemize}
\item[(i)] the $H_i$-action on $G/P$ is equalized and has isolated extremal fixed-points; 
\item[(ii)] $P=P_+$ and the grading of $\fg$ defined by $H_i$ is short and balanced. 
\end{itemize}
\end{corollary}

\begin{proof}
We have already noted that the $H_i$-action is equalized on every $G/P$ if and only if it is given by a short grading of $\fg$ (Proposition \ref{prop:Hi}). 
On the other hand, consider the $H_i$-action on $G/P$. From Section \ref{sssec:fpc} we know that the sink of the action is isolated is and only if $G^\perp\subset P$, which holds if and only if $P=P_+$. If, moreover, the action is balanced, then by definition, the point $w_0P_+$ (which is always contained in the source of the $H_i$-action on $G/P_+$) will be an isolated fixed-point, as well. Finally, assume that the $H_i$-action on $G/P_+$ is equalized and has $w_0P_+$ as an isolated source. From the description of the equalized actions of Section \ref{sec:equalGrass} for the classical cases, \cite[Proposition 5.2]{Fra1} for the $\DE_6$-varieties and \cite[Theorem 8.9]{WORS1} for the $\DE_7$-variety, we get that our conditions are satisfied precisely in the cases listed in Table \ref{tab:balshort}, i.e. in the cases in which the grading is balanced.  
\end{proof}

We will now study the birational map associated with the $\C^*$-action induced by a balanced short grading as above, linking it with an inversion map through the map $L_{w_0}$.

The unipotent radicals of $P_\pm$ -- that we denote by $G_{\pm}\subset P_{\pm}$ -- have $\fg_{\pm}$ as Lie algebras. Being $G$ the adjoint group of a semisimple Lie algebra, $G_\pm\subset G$ consists only of nilpotent elements and then it is known that the exponential maps $\exp:\fg_\pm \to G_\pm$ are isomorphisms of varieties (see \cite[Proposition 1.2]{Huy}).

\begin{lemma}
	The composition of $\exp:\fg_-\to G_-\subset G$ with the projection onto $G/P_+$ is an open immersion, that we denote by $\ex:\fg_-\hookrightarrow G/P_+$. 
	Its image, which we still denote by $\fg_-$, is an open neighborhood of $eP_+\in G/P_+$.
\end{lemma}

\begin{proof}
	First of all, $\fg_-$ and $G/P_+$ have the same dimension.
	At this point, denoted by $\pi: G \to G/P_+$ the quotient map, let $U \subset G$ be an open set. By definition, $\pi(U)$ is open in $G/P_+$ if and only if $\pi^{-1}(\pi(U))$ is open in $G$:
	\[
	\pi(U)=\{uP: u \in U\}, \qquad \pi^{-1}(\pi(U))=\{up: u \in U,\, p \in P_+\}=U \cdot P_+.
	\]
	By \cite[p. 86]{BT}, $G_- \cdot P_+$ is open in $G$, hence $\pi(G_-)$ is open in $G/P_+$. It remains to prove that $\pi(G_-) \simeq G_-$, but this follows from the fact that, since $G_-$ is the unipotent radical of the opposite parabolic group of $P_+$ (see \cite[p. 88]{BT}), we have $G_- \cap P_+=\{e\}$.
	We have obtained that $\fg_- \simeq G_- \simeq \pi(G_-) \subset G/P_+$ as an open subset.
\end{proof}

Furthermore, denoting by $L_t:G/P_+\to G/P_+$ the left multiplication with $t$, for every $t\in H_i$, we have a commutative diagram:
\[
\xymatrix{\fg_-\ar@{_(->}+<0pt,-12pt>;[d]_\ex \ar[r]^{t^{-1}} &\fg_- \ar@{_(->}+<0pt,-12pt>;[d]_\ex\\
G/P_+ \ar[r]^{L_t} & G/P_+}
\] 
where the upper horizontal map is the multiplication with $t^{-1}$. In other words, $\fg_-$ can be identified with a $\C^*$-invariant open subset of $G/P_+$, and the action of $\C^*$ on $\fg_-$ is the inverse of the homothetical action. 

On the other hand, we have another open immersion:
\[
	\ex':\fg_+ \lra G/P_+, \qquad x \longmapsto \exp(x)w_0P_+,
\]
whose image is a $\C^*$-invariant open neighborhood (denoted by $\fg_+$) of $w_0P_+\in G/P_+$, that fits into the following commutative diagram:
\[
\xymatrix{\fg_-\ar@{_(->}+<0pt,-12pt>;[d]_\ex\ar[r]^{\Ad_{w_0}}&\fg_+\ar@{_(->}+<0pt,-12pt>;[d]_{\ex'}\\
G/P_+\ar[r]^{L_{w_0}}&G/P_+.}
\]
Intersecting $\fg_-$ and $\fg_+$ inside $G/P_+$, and restricting $\Ad_{w_0}$ to $\Ad_{w_0}^{-1}(\fg_-\cap \fg_+)$ we get a birational map that we denote 
\[
\jmath:\fg_-\dashrightarrow \fg_-.
\]
By construction, we have a commutative diagram:
\begin{equation}\label{eq:inversion}
\xymatrix@R=2mm{&&\fg_-\ar[dd]^{\Ad_{w_0}}\\\fg_-\ar@{-->}[rru]^{\jmath}\ar@{-->}[rrd]_{\psi_a}&&\\&&\fg_+.}
\end{equation}
We note that, by Lemma \ref{lem:balanced} and the fact that the inclusions $\fg_\pm\subset G/P_+$ are $\C^*$-invariant, $\jmath$ satisfies:
\begin{equation}\label{eq:jC*}
\jmath(tx)=t^{-1}x, \mbox{ for all }t\in H_i,\,\, x\in \fg_-.
\end{equation}
so that the inversion map can be thought of as a birational automorphism counterpart of the birational map $\psi_a$ induced by the $H_i$-action (see Section \ref{ssec:basic}).

The following statement is a consequence of \cite[2.21]{Spr73}.

\begin{proposition}\label{prop:inverseJ}
With the above notation, if the short grading on $\fg$ determined by $\sigma_i:\Mo(H)\to \Z$ is balanced, then there exists a unique structure of Jordan algebra on $\fg_-$ having $\jmath$ as an inverse map.
\end{proposition}

\begin{proof}
Following \cite{Spr73}, the statement follows from the fact that the map $\jmath$ defines a {\em J-structure} on $\fg_-$, that determines completely a structure of Jordan algebra on $\fg_-$ whose inverse map is $\jmath$. Furthermore, \cite[2.21]{Spr73} tells us that in our case one needs to check three conditions (A,B,C). Condition (A) follows from the fact that our grading is equalized, and (B) is precisely  Equation (\ref{eq:jC*}). The last condition, (C), translated into our notation says that $G_0$ acts with an open orbit on $\fg_-$, and this follows by direct application of \cite[Theorem~2.1]{Tev05}.
\end{proof}

This concludes the proof of Theorem \ref{thm:Jordan}.




\section{Proof of Theorem \ref{thm:unbalanced} and Theorem \ref{thm:E7}}\label{sec:proofThm1.2}



We have already seen in Section \ref{sec:j-str} that, when a $\C^*$-action on an RH variety $G/P$ is determined by a balanced short grading (and its sink and source are isolated), the induced birational map (which is a Cremona transformation) is determined, up to composition with a projectivity, by the inverse map of a \emph{unique} Jordan algebra structure on the tangent space of $G/P$. 

Suppose now that the sink and source are not isolated, then the birational map $\psi:\cG_-\dashrightarrow\cG_+$ induced by the action has as domain and codomain (which are, by the balancedness hypothesis, isomorphic) the projectivizations of two homogeneous vector bundles on two RH varieties of Picard number one. In particular, $\cG_{\pm}$ have Picard number two. Moreover, it has been proved in \cite[Section 4 and Corollary 4.12 (iii) in particular]{WORS3} that $\cG_{\pm}$ are Mori Dream Spaces, isomorphic in codimension one. More precisely, following \cite{WORS3}, we may state:  

\begin{proposition}\label{prop:unique}
Let $G$ be a simple algebraic group, whose Lie algebra $\fg$ admits a short and balanced grading. Let $X=G/P$ be an RH variety of Picard number one, endowed with the $\C^*$-action associated to the grading. Assume moreover that the sink and the source $\cG_{\pm}$ are not isolated fixed-points. Then:
\begin{itemize}[leftmargin=*]
\item The quotients $\cG_{\pm}$ are Mori dream spaces, corresponding to two different nef chambers of $\overline{\Mov(\cG_-)}=\overline{\Mov(\cG_+)}$.
\item There are as many nef chambers of $\overline{\Mov(\cG_-)}$ as $\delta:=\mu_L(Y_+)-\mu_L(Y_-)$, where $L$ denotes the ample generator of $\Pic(X)$. Each nef chamber corresponds to a geometric quotient of $X$ by the action of $\C^*$, more precisely
$$
\cG_k:=\Proj\left(\bigoplus_{m\in 2\Z_{\geq 0}}\HH^0\big(X,L^{m}\big)_{((2k-1)m/2)}\right),\qquad k=1,\dots, \delta.
$$
\item We have $\cG_1=\cG_-$ and $\cG_\delta=\cG_+$. Each of the nef cones of $\cG_{\pm}$ contains an extremal ray of $\Mov(\cG_-)$.
\item The map $\psi:\cG_-\dashrightarrow\cG_+$ can be factorized as the sequence of the birational maps:
$$
\cG_1=\cG_-\dashrightarrow \cG_2 \dashrightarrow\dots \dashrightarrow \cG_\delta=\cG_+,
$$
which are Atiyah flips corresponding to wall crossings in $\overline{\Mov(\cG_-)}$.
\end{itemize}
\end{proposition}

In particular, this immediately implies the following:

\begin{corollary}\label{cor:unique}
In the situation of Proposition \ref{prop:unique}, the birational map $\psi:\cG_-\dashrightarrow\cG_+$ is a small $\Q$-factorial modification, uniquely determined by $\cG_{\pm}$. 
\end{corollary}

We will finish the Section with the proofs of Theorem \ref{thm:unbalanced} and Theorem \ref{thm:E7}, which is done by determining the normal bundles of $Y_\pm$ in $X$ in each case. 


\subsection{Proof of Theorem \ref{thm:unbalanced}}\label{ssec:1.2classic} 


By Corollary \ref{cor:unique}, the map $\psi$ is in each case determined by the varieties $\cG_\pm$, which we will describe in each case. We will proceed case-by-case according to Table \ref{tab:balshort}. 
Besides, we have included here some remarks, in which we provide further geometric features of $\psi$.

\subsubsection{$H_n$-action on Grassmannians} \label{ssec:caseA}

We consider the equalized and balanced $H_n$-action on the Grassmannian $\DA_{2n-1}(k)$.
We refer to Section \ref{ssec:An} for the notations and results therein contained.
In particular, given a the $2n$-dimensional vector space $V$, the $H_n$-action provides a decomposition $V=V_- \oplus V_+$ such that $\dim V_\pm=n$.
The extremal fixed-point components of the action on $\DA_{2n-1}(k)$ are:
\begin{equation}\label{eq:sinkSourceGrassEB}
	Y_\pm = \left\{[W] \in \P \left( \bigwedge^k V_\pm\right):W \subset V_\pm\right\}\simeq \DA_{n-1}(k);
\end{equation}
they parametrize $k$-dimensional vector subspaces contained in $V_\pm$, respectively.
We denote by $\cS_\pm,\cQ_\pm$ the universal bundles over $Y_\pm$ of rank $k$ and $n-k$, respectively.

If $k=n$, $\psi$ is a Cremona transformation (see Section \ref{sec:j-str}), which can be explicitly described (see \cite[Section 4]{Thaddeus}, \cite[Section 4]{WORS5}) as the projectivization of the inversion of linear maps. More generally:

\begin{proposition}\label{prop:EBgrass}
	With the notation as above, consider the equalized and balanced $H_n$-action on the Grassmannian $\DA_{2n-1}(k)$. Then $Y_\pm\simeq \DA_{n-1}(k)$ and 
	\[
	\cN_{Y_\pm|\DA_{2n-1}(k)}\simeq \cS_\pm^\vee \otimes V_\mp.
	\]
\end{proposition}

\begin{proof}
In order to compute the normal bundles $\cN_{Y_\pm|\DA_{2n-1}(k)}$, we restrict the universal bundles $\cS,\cQ$ of $\DA_{2n-1}(k)$ on the fixed-point components $Y_\pm$, obtaining:
\begin{align*}
	&\cS|_{Y_\pm}\simeq\cS_\pm, \\
	&\cQ|_{Y_\pm}\simeq\cQ_\pm \oplus \left( V_\mp \otimes \cO_{Y_\pm}\right),
\end{align*}
where $\cS_\pm,\cQ_\pm$ are the universal bundles of $Y_\pm$.
Consider the short exact sequences
\[
\shse{T_{Y_\pm}}{T_{\DA_{2n-1}(k)}|_{Y_\pm}}{\cN_{Y_\pm|\DA_{2n-1}(k)}};
\]
together with the equalities
\begin{align*}
	&T_{Y_\pm} \simeq \cS_\pm^\vee \otimes \cQ_\pm,\\
	&T_{\DA_{2n-1}(k)}|_{Y_\pm} \simeq \left(\cS_\pm^\vee \otimes \cQ_\pm\right) \oplus \left(\cS_\pm^\vee \otimes \left(V_\mp \otimes \cO_{Y_\pm}\right)\right);
\end{align*}
then we get
\[
\cN_{Y_\pm|\DA_{2n-1}(k)} \simeq \cS_\pm^\vee \otimes V_\mp. \qedhere
\]
\end{proof}

\begin{remark}\label{rem:thaddeus}
	Following \cite[Section 4]{Thaddeus}, for every $u\leq k$ we define $\Sec_u(\cS_{\pm},V_\mp)$ to be the variety of secant $(u-1)$-dimensional linear spaces to the relative Segre embedding 
	\[
	\P\left(\cS_{\pm}^\vee\right) \times \P\left(V_\mp\right) \subset \P\left (\cS_{\pm}^\vee \otimes V_\mp\right).
	\]
	Thaddeus describes such a variety as the space of linear maps $f: W_\pm \to  V_\mp$ (where $W_\pm \subset V_\pm$  is a linear subspace of dimension $k$) of rank up to $u \le k$. 

	Our map $\psi$ is then defined on the open set 
	\[
	\P\left (\cS_{\pm}^\vee\otimes V_\mp\right) \setminus \Sec_{k-1}(\cS_{\pm},V_\mp),
	\]
	which corresponds to the linear maps $f: W_\pm \to  V_\mp$ of maximal rank.

	Assume that $f: W_- \to V_+$ is an injective linear map from a $k$-dimensional linear subspace $W_-\subset V_-$ and denote $W_+:=f(W_-)$.
	Consider the inverse $f^{-1}:W_+ \to  W_-$ and compose it with the inclusion $W_-\hookrightarrow V_-$, to a map $\bar f: W_+ \to  V_-$. 
	Then $\psi$ can be described as follows:
	\[
	\psi:\P\left( \cS_{-}^\vee\otimes V_+ \right) \dashrightarrow \P\left( \cS_{+}^\vee \otimes V_- \right), \qquad [f] \longmapsto [\bar f].
	\]
	Note that, if $k=1$, then $\cS_\pm^\vee= \cO_{Y_\pm}(1)$. So $\P(\cO_{Y_\pm}(1) \otimes V_\mp)=\P(V_-) \times \P(V_+)$, and the map $\psi$ is the identity. 
\end{remark}

\subsubsection{$H_n$-action on isotropic Grassmannians}\label{sssec:HnsympOrtho}

We merge the cases of the equalized and balanced $H_n$-actions on the symplectic Grassmannian $X \simeq \DC_n(k)$ or the even orthogonal Grassmannian $X \simeq \DD_n(k)$, with $n$ even if $k=n-1,n$.
As shown in Section \ref{ssec:Cn} and Section \ref{sssec:HnDn}, which we refer for the notations and the results therein, the fixed-point components of the action are RH varieties of type $\DA_{n-1}$.

With the exception of the case $X\simeq \DD_n(n-1)$, $n$ even, the sink and the source of the $H_n$-action on $X$, as in \eqref{eq:sinkSourceGrassEB}, are Grassmannians parametrizing $k$-dimensional linear subspaces of $V_\pm$, respectively:
\[
Y_\pm =\left\{[W] \in \P \left( \bigwedge^k V_\pm\right):W \subset V_\pm\right\} \simeq \DA_{n-1}(k).
\]
As in Section \ref{ssec:caseA}, $Y_\pm$ support two universal bundles, of rank $k$ and $n-k$ respectively, that we denote by $\cS_{\pm},\cQ_{\pm}$. Moreover, the isomorphism $$V_+ \simeq V_-^\vee$$ (see Section \ref{sssec:univBundles}) allows us to identify $Y_\pm$ with the Grassmannians of $(n-k)$-dimensional vector subspaces of $V_\mp$. Then, restricting Diagram \eqref{eq:univBD} to $Y_-$, for instance, we get:
\begin{equation}\label{eq:univBD3}
\xymatrix@R=10mm{\cS_- \ar@{=}[d] \ar[r] & (V_-\otimes \cO_{Y_-})\oplus\cQ_-^\vee \ar[d] \ar[r] & \cK_{|Y_-}  \ar[d] \\ 
\cS_- \ar[r] \ar[d] & (V_-\oplus V_-^\vee) \otimes \cO_{Y_-} \ar[d] \ar[r] &(V_-^\vee \otimes \cO_{Y_-})\oplus \cQ_- \ar[d]\\
0 \ar[r] & \cS_-^\vee \ar@{=}[r] & \cS_-^\vee}
\end{equation} 
so that 
\begin{equation}\label{eq:KQQdual}
\cK_{|Y_-}\simeq \cQ_-\oplus\cQ_-^\vee.
\end{equation}

In the case $X\simeq \DD_n(n-1)$, $n$ even, the sink and the source of the action are projective spaces $\DA_{n-1}(n-1)$, parametrizing $(n-1)$-dimensional subspaces of $V_-$ and $V_+$, respectively. The restriction of Diagram \eqref{eq:univBD} to $Y_-$ provides:
\begin{equation}\label{eq:univBD2}
\xymatrix@R=10mm{\cS_{|Y_-} \ar@{=}[d] \ar@{=}[r] & \cS_{|Y_-} \ar[d] \ar[r] & 0  \ar[d] \\ 
\Omega_{Y_-}(1)\oplus  \cO_{Y_-}(-1) \ar[r] \ar[d] & (V_-\oplus V_-^\vee) \otimes \cO_{Y_-} \ar[d] \ar[r] &\cO_{Y_-}(1)\oplus T_{Y_-}(-1) \ar@{=}[d]\\
0 \ar[r] & \cS_{|Y_-}^\vee \ar@{=}[r] & \cS_{|Y_-}^\vee.}
\end{equation} 

We will study now the normal bundles of $Y_\pm$ in $X$ (which completely determine the induced birational map $\psi: \P \left(\cN_{Y_-|X}\right) \dashrightarrow \P \left( \cN_{Y_+|X}\right)$ by Corollary \ref{cor:unique}).
The case $k=n$ (that is, the case in which $\psi$ is a Cremona transformation) has been explicitly studied in \cite[Sections 4.3 and 4.4]{WORS5} (see also \cite[Appendix 10 and 11]{Thaddeus}) which show that:

\begin{itemize}[leftmargin=*]
	\item If $X \simeq \DC_n(n)$ is a Lagrangian Grassmannian, then the birational map \eqref{eq:birmap} is the inversion of symmetric tensors:
\[
\psi:\P\left(S^2 V_+\right) \dashrightarrow \P\left(S^2 V_-\right)\simeq\P\left(S^2 V_+^\vee\right).
\]
\item If $X \simeq \DD_n(n)$, $n$ even, is a spinor variety, then the birational map \eqref{eq:birmap} is the inversion of skew-symmetric tensors:
\[
\psi:\P\left(\bigwedge^2 V_+\right) \dashrightarrow \P\left(\bigwedge^2 V_-\right)\simeq \P\left(\bigwedge^2 V_+^\vee\right).
\]
\end{itemize}

To describe the normal bundles of $Y_\pm$ in $X$ in the remaining cases, we will consider separately the case in which $X$ parametrizes non-maximal isotropic subspaces of $V$ and the case $X \simeq \DD_n(n-1)$, $n$ even. We start by proving the following. 

\begin{proposition}\label{prop:normalHn}
	With the notation as above, consider the equalized and balanced $H_n$-action on $X \simeq \DC_n(k)$ with $k<n$ or $X \simeq \DD_n(k)$ with $k<n-1$.
	Then $Y_\pm\simeq \DA_{n-1}(k)$ and the normal bundles $\cN_{Y_\pm|X}$ are non-trivial extensions 
	\[
	\shse{\left(\cS_{\pm} \otimes \cQ_{\pm}\right)^\vee}{\cN_{Y_\pm|X}}{\cC_{\pm}},
	\]
	where
	\[
	\cC_{\pm}=\begin{cases}
	S^2\cS_{\pm}^\vee &\mbox{if }X \simeq \DC_{n}(k),\\\bigwedge^2\cS_{\pm}^\vee &\mbox{if }X \simeq \DD_{n}(k).\end{cases}
	\]
\end{proposition}

\begin{proof}
	Combining the fact that $T_{Y_\pm}\simeq \cS_{\pm}^\vee \otimes \cQ_{\pm}$ with \eqref{eq:tangentBD}, we obtain:
	\begin{equation}\label{eq:normHn}
	\xymatrix{ \cS_{\pm}^\vee \otimes \cQ_{\pm} \ar[d] \ar[r]^\simeq & T_{Y_\pm} \ar[d] & \\
	\cS_{\pm}^\vee \otimes \cK|_{Y_\pm} \ar[d] \ar[r] & T_{X}|_{Y_\pm} \ar[d] \ar[r] & \cC_{\pm} \ar@{=}[d] \\
	\cS_{\pm}^\vee \otimes \cK|_{Y_\pm}/\cQ_{\pm} \ar[r] & \cN_{Y_\pm |X} \ar[r] & \cC_{\pm}.
	}
	\end{equation}
	We conclude by noting that $\cK|_{Y_\pm}/\cQ_{\pm}=\cQ_{\pm}^\vee$, which follows from the fact that $\Hom(\cQ,\cQ^\vee)=0$. 
\end{proof}

Finally we consider the case of the $H_n$-action on $X \simeq \DD_n(n-1)$.

\begin{proposition}\label{prop:lastOne}
	With the notations as above, consider the equalized and balanced $H_n$-action on $\DD_n(n-1)$, $n$ even. Then $Y_\pm\simeq \DA_{n-1}(n-1)$ are projective spaces, and 
	\[
	\cN_{Y_\pm|\DD_n(n-1)}\simeq \bigwedge^2 T_{Y_\pm}(-2).
	\]
\end{proposition}

\begin{proof}
	We will do the proof in the case of $Y_-$, being the case of $Y_+$ analogous. 
	By \eqref{eq:tangentBD} we have that $T_{\DD_n(n-1)}\simeq \bigwedge^2 \cS^\vee$. Using Diagram \eqref{eq:univBD2} we then get:
	\[
	T_{\DD_n(n-1)}|_{Y_-}\simeq \bigwedge^2 \cS^\vee |_{Y_-}\simeq \bigwedge^2\left( \cO_{Y_-}(1) \oplus T_{Y_-}(-1)\right)\simeq T_{Y_-} \oplus \bigwedge^2T_{Y_-}(-2).
	\]
	Then we have a short exact sequence:
	\begin{equation}\label{eq:splitSequence}
			\shse{T_{Y_-}}{T_{Y_-} \oplus \bigwedge^2 T_{Y_-}(-2)}{\cN_{Y_-|\DD_n(n-1)}}.
	\end{equation}
	We claim that $\Hom(T_{Y_-}, \bigwedge^2 T_{Y_-}(-2))=0$, so that, in particular we conclude that  $\cN_{Y_\pm|\DD_n(n-1)}\simeq \bigwedge^2 T_{Y_\pm}(-2)$. If there was a non-trivial morphism $T_{Y_-} \to \bigwedge^2 T_{Y_-}(-2)$, twisting it with $\cO_{Y_-}(-1)$, we would obtain a non-trivial map
	\[
	T_{Y_-}(-1) \lra \bigwedge^2 T_{Y_-}(-3).
	\]
	Since $T_{Y_-}(-1)$ is globally generated, this would imply that $\HH^0(Y_-,\bigwedge^2 T_{Y_-}(-3))\neq 0$.
	On the other hand, we have that
	\[
	\bigwedge^2 T_{Y_-}(-3) \simeq \bigwedge^{n-3} \Omega_{Y_-}(n-3)
	\]
	and that $H^0(Y_-, \bigwedge^{n-3} \Omega_{Y_-}(n-3))=0$ by Bott formula (cf. \cite[p.8]{OSS}), leading to a contradiction.	
\end{proof}

\subsubsection{$H_1$-action on orthogonal Grassmannians}\label{sssec:H1ortho}

We consider here together the cases of the equalized and balanced $H_1$-actions on the orthogonal Grassmannian $X \simeq \DB_n(k)$ or $X \simeq \DD_n(k)$. 
We refer to Section \ref{ssec:Bn} and Section \ref{sssec:H1Dn} for the notation and preliminary results.
In particular, the fixed-point components are RH varieties of type $\DB_{n-1}$ or $\DD_{n-1}$, respectively (that can be identified with orthogonal Grassmannians of isotropic subspaces in $V_0\subset V$), and the possible pairs $(Y_-,Y_+)$ are:
\begin{align*}
	&\left(\DB_{n-1}(k-1),\DB_{n-1}(k-1)\right) & & \text{if }X \simeq \DB_n(k),\\
	&\left(\DD_{n-1}(k-1),\DD_{n-1}(k-1)\right) & & \text{if }X \simeq \DD_n(k) \text{ with }k \le n-2,\\
	&\left(\DD_{n-1}(n-2),\DD_{n-1}(n-1)\right) & & \text{if }X \simeq \DD_n(n-1),\\
	&\left(\DD_{n-1}(n-1),\DD_{n-1}(n-2)\right) & & \text{if }X \simeq \DD_n(n).
\end{align*}
As orthogonal Grassmannians of $(k-1)$-dimensional isotropic subspaces of $V_0$, $Y_\pm$ support two universal bundles $\cS_{\pm},\cQ_{\pm}$ of rank $k-1$ and $\dim V_0-(k-1)$, respectively, fitting in short exact sequences
\[
\shse{\cS_\pm}{V_0 \otimes \cO_{Y_\pm}}{\cQ_\pm}.
\]

The case $k=1$ gives us the Cremona transformation as described in Section \ref{sec:j-str}; we will provide later a projective description of the map $\psi$ obtained in this case (cf. Remark \ref{rem:H1k1}).
More generally, the following statement holds:

\begin{proposition}\label{prop:normalBD}
	With the notation as above, consider the equalized and balanced $H_1$-action on the orthogonal Grassmannian $X \simeq \DB_n(k)$ or $X \simeq \DD_n(k)$. Then 
	\[
	\cN_{Y_\pm|X}\simeq \cQ_\pm.
	\]
\end{proposition}

\begin{proof}
	We will prove the statement for the sink $Y_-$, being the case of $Y_+$ analogous. 
	We start by writing $Y_-$ as the RH $G^\perp$-variety $G^\perp/P^\perp$ (Section \ref{sssec:fpc}); then, since both $\cN_{Y_-|X}$ and $\cQ_{-}$ are homogeneous $G^\perp$-bundles (Lemma \ref{lem:normals}), it is enough to show that their fibers at $eP^\perp$ are isomorphic as $P^\perp$-modules. 
	
	On one hand, Lemma \ref{lem:normals} tells us that
	\[
	\cN_{Y_-|X,\, eP^\perp}=N_-:=\bigoplus_{\beta \in S} \fg_\beta
	\]
	where $S=\left\{\beta \in \Phi^-: \sigma_1(\beta)<0,\, \sigma_k(\beta)<0\right\}$. Using the notation of Bourbaki (cf. \cite[Planche II,IV]{Bourb}), there exists a basis $\{\epsilon_i,\,\,i=1,\dots n\}$ of $\Z\Phi\otimes_\Z\R$ such that the positive simple roots of $\fg$ can be written as:
	\[
		\alpha_i=\epsilon_i-\epsilon_{i+1} \text{ for } i<n \text{ and }\alpha_n= \begin{cases}
 		\epsilon_n & \text{if }\fg=\DB_{n}, \\
 		\epsilon_{n-1}+\epsilon_n & \text{if }\fg=\DD_{n},
	\end{cases}
	\]
	hence we can rewrite $S$ as
	\[
	S=\left\{-\epsilon_1+\epsilon_j\right\}_{j=k+1}^n \cup \left\{-\epsilon_1-\epsilon_j\right\}_{j=2}^n \cup 
	\begin{cases}
 	\left\{-\epsilon_1\right\} & \text{if }\fg=\DB_{n} \\
 	\emptyset & \text{if }\fg=\DD_{n}
	\end{cases}
	\]
	in the case $X \not\simeq \DD_n(n-1)$ and 
	\[
	S=\left\{-\epsilon_1+\epsilon_n\right\} \cup \left\{-\epsilon_1-\epsilon_j\right\}_{j=2}^{n-1}.
	\]
	otherwise.

	Since the first fundamental weight $\omega_1$ is $\epsilon_1$, by Remark \ref{rem:projectionM(H)} the $H^\perp$-weights of the action are 
	\[
	\left\{\epsilon_j\right\}_{j=k+1}^n \cup \left\{-\epsilon_j\right\}_{j=2}^n \cup 
	\begin{cases}
 	\{0\} & \text{if }\fg=\DB_{n} \\
 	\emptyset & \text{if }\fg=\DD_{n}
 	\end{cases}
	\]
	when $X \not\simeq \DD_n(n-1)$ and 
	\[
		\left\{\epsilon_n\right\} \cup \left\{-\epsilon_j\right\}_{j=2}^{n-1}
	\]
	otherwise.
	
	On the other hand, consider the basis $\cB=\{e_i\}$ of $V$ introduced in Sections \ref{ssec:Bn} and Section \ref{ssec:Dn}. Then $\cB \setminus \left\{e_0,e_n\right\}$ is a basis for $V_0$; in particular, for $i <n$, the $H^\perp$-weight of $e_i$ is $\epsilon_{i+1}$, the $H^\perp$-weight of $e_{n+i}$ is $-\epsilon_{i+1}$ and, if $\fg=\DB_n$, the $H^\perp$-weight of $e_{2n}$ is $0$.
	We denote
	\[
	V_0 \supset W:=\begin{cases}\langle e_1,\ldots,e_{k-1}\rangle & \text{if }X\not\simeq \DD_{n}(n-1),\\
		\langle e_1,\ldots,e_{n-2},e_{2n-1}\rangle & \text{if } X\simeq \DD_{n}(n-1),
	\end{cases}
	\]
	we have that 
	\[
	\cQ_{-}=G^\perp \times^{P^\perp} V_0/W.
	\]
Then a straightforward computation shows that the decomposition of $V_0/W$ on $H^\perp$-eigenspaces is the same as the one obtained above. 
\end{proof}

\begin{remark}\label{rem:H1k1}
In the case $k=1$, the induced map $\psi$ is a Cremona transformation determined by the inversion in a Jordan algebra (see Section \ref{sec:j-str}). 
Furthermore, $\psi:\P(\fg_-)\to\P(\fg_+)$ is a linear isomorphism that can be projectively described as follows. 
The symmetric bilinear form defining the Lie algebra $\fg$ determines a quadric $Q=\cD(1)$ ($\cD=\DB_n$ or $\DD_n$) in the projectivization of the standard representation $V=V(\omega_1)$ of $\fg$. 
Then $\fg_\pm$ correspond to the tangent spaces of $Q$ at sink and source $Y_\pm$, which are isolated points. 
Consider a tangent direction $[v_-]\in \P(\fg_-)$, it determines a line $\ell_-$ passing by $Y_-$.
The plane $\ell_-+Y_+$ intersects $Q$ on a ($\C^*$-invariant) conic, which is smooth at $Y_+$; the tangent direction to the conic at $Y_+$ is $\psi([v_-])$.
\end{remark}

Let us now describe projectively also the cases in which $k>1$.

\begin{remark}\label{rem:BD}
In the cases in which $X$ is a spinor variety, i.e. $X \simeq \DB_n(n), \DD_n(n-1),\DD_n(n)$, there are no inner fixed-point components, hence the induced birational transformation $\psi$ in the statement above is an isomorphism.   
In the remaining cases, $\psi:\P( \cQ_{-}) \dashrightarrow \P( \cQ_{+})$ is an Atiyah flip (cf. \cite{WORS1, WORS3}), that can be described projectively as follows.

Again, we identify $Y_\pm $ with two orthogonal Grassmannians of $(k-1)$-dimensional isotropic subspaces of $V_0$. Then the bundle $\P\left(\cQ_{-}\right)$ can be identified with the family of pairs $(W_-,W_0)$, where $W_-\in Y_-$, and $W_0\subset V_0$ is a subspace of dimension $k$ containing $W_-$. 
A similar description applies to $\P \left(\cQ_+\right)$. 

Given two pairs $(W_-,W_0)\in \P \left(\cQ_-\right)$, $(W_+,W_0')\in \P \left(\cQ_+\right)$, we have that $\psi$ sends $(W_-,W_0)$ to $(W_+,W_0')$ if $W_0=W_0'$ and the set of isotropic vectors of $W_0$ is precisely $W_-\cup W_+$.
The maps $\psi$, $\psi^{-1}$ are not defined, respectively, in the sets
\[
\Lambda_\pm:=\left\{(W_\pm,W_0)\in \P(\cQ_{\pm}):W_0\subset V_0 \mbox{ isotropic}\right\}\subset \P(\cQ_{\pm}).
\]
Denote $Z\simeq \DA_{\dim V_0-1}(k)$, the Grassmannian of $k$-dimensional subspaces of $V_0$. We have an obvious (small) contraction $\P (\cQ_\pm)\to Z$, sending $(W_\pm,W_0)$ to $W_0$; the image of $\Lambda_\pm$ via this contraction, that we denote by $\Lambda \subset Z$, is the orthogonal Grassmannian of isotropic $k$-dimensional subspaces of $V_0$. Then the map $\psi$ is precisely a flip for these two contractions:
\[
\xymatrix@R=6mm@C=20pt{
\P(\cQ_-)\ar@/^2pc/@{-->}[rrrr]_\psi \ar[rrdd]_{\text{\tiny small}} &\,\,\Lambda_- \ar@{_(->}[l]\ar[rd]& &\Lambda_+ \,\,\ar@{^(->}[r]\ar[ld]& \P(\cQ_+).\ar[lldd]^{\text{\tiny small}}\\
& &\Lambda \ar@{^(->}+<0pt,-12pt>;[d] & &\\
& &Z& &
}
\]
\end{remark}


\subsection{Proof of Theorem \ref{thm:E7}}\label{ssec:E7} 


To complete the description of the birational maps associated to a short and balanced grading, we are left with the case of $\DE_7$ and the grading given by the height map $\sigma_7$. The corresponding $H_7$-action on the variety $\DE_7(7)$ has isolated sink and source, and the associated birational map $\psi:\P(T_{\DE_7(7),Y_-})\dashrightarrow \P(T_{\DE_7(7),Y_+})$ is a Cremona transformation $\P^{26}\dashrightarrow \P^{26}$. Following \cite[Section~8]{WORS1}, this is the special Cremona transformation whose exceptional locus is the Cartan variety $\DE_6(6)\simeq \DE_6(1)$, which is the $16$-dimensional Severi variety. 

Let us now consider the induced $H_7$-action on varieties $\DE_7(k)$, $k<7$. As in the cases of classical type, the induced birational map is completely determined by the normal bundles $\cN_{Y_\pm|\DE_7(k)}$ (cf. Corollary \ref{cor:unique}). 
The extremal fixed-point components of the action $Y_{\pm}\subset \DE_7(k)$, which are RH $\DE_6$-varieties (see Table \ref{tab:E7} below), can be obtained by applying directly the arguments in Section \ref{sssec:fpc}. 
\begin{table}[h!!]
\begin{tabular}{|c||c|c|c|c|}
\hline
$k$&$Y_-$&$Y_+$&$\dim(Y_{\pm})$&$\rank(N_{Y_{\pm},\DE_7(k)})$\\\hline\hline
$1$&$\DE_6(1)$&$\DE_6(6)$&$16$&$17$\\\hline
$2$&$\DE_6(2)$&$\DE_6(2)$&$21$&$21$\\\hline
$3$&$\DE_6(3)$&$\DE_6(5)$&$25$&$22$\\\hline
$4$&$\DE_6(4)$&$\DE_6(4)$&$29$&$24$\\\hline
$5$&$\DE_6(5)$&$\DE_6(3)$&$25$&$25$\\\hline
$6$&$\DE_6(6)$&$\DE_6(1)$&$16$&$26$\\\hline
\end{tabular}
\caption{Extremal fixed-point components of the $H_7$-action in $\DE_7(k)$, $k<7$.\label{tab:E7}}
\end{table}

Note that in each case, the markings of the Dynkin diagram $\DE_6$ for the sink and the source are symmetric with respect to the nontrivial automorphism of the $\DE_6$ diagram (cf. Remark \ref{rem:Y-Y+marked}) that we denote by $s$:
\[
s(1)=6,\quad s(2)=2,\quad s(3)=5,\quad s(4)=4,\quad s(5)=3,\quad s(6)=1.
\]

\begin{notation}
	 In order to include the case of $\DE_7(7)$, we define $s(7):=0$.
\end{notation}

The variety $\DE_6(6)$ can be embedded via $L$, the generator of its Picard group, into the projectivization of the ($27$-dimensional) representation $V(\omega_6)$, so that $V(\omega_6)^\vee=H^0\left(\DE_6(6),L\right)$. 
On the other hand, for every $k\le 6$, the variety $\DE_6(k)$ can be thought of as a family of projective spaces in $\P(V(\omega_6))$. 
In fact, for $k=6$ this is given simply by the embedding $\DE_6(6)\subset \P\left(V(\omega_6)\right)$. For $k\neq 1,6$ the natural projections of $\DE_6$-varieties
\[
\xymatrix{ & \DE_6(k,6) \ar[ld]_{p_k} \ar[rd] \ar[rrd]^{q_k} && \\\DE_6(k) & & \DE_6(6)\ar@{^(->}[r]&\P(V(\omega_6))}
\]
present $\DE_6(k)$ as a family of $\P^5$'s, $\P^4$'s, $\P^2$'s, and $\P^1$'s, in the cases $k=2,3,4,5$, respectively. 
Finally, the variety $\DE_6(1)$ parametrizes smooth $8$-dimensional quadrics in $\DE_6(6)$; considering the linear span of these quadrics, we may also think of $\DE_6(1)$ as the parameter space of a family of ($9$-dimensional) projective subspaces of $\P(V(\omega_6))$:
\[
\xymatrix{ & \DE_6(1,6) \ar[ld]_{p_1} \ar[rrd]^{q_1} && \\\DE_6(1) & & &\P(V(\omega_6)).}
\]
These families of ($9,5,4,2,1$-dimensional) projective subspaces of $\P\left(V(\omega_6)\right)$ are the projectivizations of the following $\DE_6$-homogeneous vector bundles:
\[
\cS_k:={p_k}_*q_k^*\cO_{\P(V(\omega_6))}(1).
\]
This is a subbundle of the trivial vector bundle $V(\omega_6) \otimes\cO_{\DE_6(k)}$, and we denote the corresponding cokernel:
\[
\cQ_k:=\left(V(\omega_6) \otimes\cO_{\DE_6(k)}\right)/\cS_k.
\]

Analogously, if we start with $\DE_6(1) \subset \P(V(\omega_1))$, then we can describe the other $\DE_6$-homogeneous varieties of Picard number one as families of projective spaces in $\P\left(V(\omega_1)\right)$:
	\[
\xymatrix{ & \DE_6(1,k) \ar[ld]_{p'_k} \ar[rrd]^{q'_k} &&\\\DE_6(k) & & &\P\left(V(\omega_1)\right).}
\]
In particular, we have a homogeneous vector bundle $\cS'_k:=p'_{k*}{q'^*_k}
\cO_{\P(V(\omega_1))}(1)$ over $\DE_6(k)$ 
which is subbundle of the trivial vector bundle $V(\omega_1) \otimes \cO_{\DE_6(k)}$, and we denote the corresponding cokernel as $\cQ'_k$.

\begin{proposition}
	With the notations as above, consider the equalized and balanced $H_7$-action on $\DE_7(k)$. Then $Y_-\simeq\DE_6(k)$, $Y_+\simeq\DE_6(s(k))$, and 
	$$
	\cN_{Y_-|\DE_7(k)}=\cQ_k,\qquad \cN_{Y_+|\DE_7(k)}=\cQ'_{s(k)}.
	$$
\end{proposition}
\begin{proof}[Sketch of the proof]
The proof is analogous to the one in Proposition \ref{prop:normalBD}. The homogeneous bundles $\cQ_k$, $\cN_{Y_-|\DE_7(k)}$ (respectively $\cQ'_{s(k)}$, $\cN_{Y_+|\DE_7(k)}$) are completely determined by their fibers at a point, which are $P^\perp$-modules. One can check that these modules are isomorphic by computing their $H^\perp$-weights, by using, for instance, by using {\tt SageMath} software. 
\end{proof}

\bibliographystyle{plain}
\bibliography{bibliomin}

\begin{thebibliography}{10}

\bibitem{BB}
Andrzej Bia{\l}ynicki-Birula.
\newblock Some theorems on actions of algebraic groups.
\newblock {\em Ann. of Math. (2)}, 98:480--497, 1973.

\bibitem{BT}
Armand Borel and Jacques Tits.
\newblock Groupes r\'{e}ductifs.
\newblock {\em Inst. Hautes \'{E}tudes Sci. Publ. Math.}, 27:55--150, 1965.

\bibitem{Bourb}
Nicolas Bourbaki.
\newblock {\em \'{E}l\'ements de math\'ematique. {F}asc. {XXXIV}. {G}roupes et
  alg\`ebres de {L}ie. {C}hapitre {IV}: {G}roupes de {C}oxeter et syst\`emes de
  {T}its. {C}hapitre {V}: {G}roupes engendr\'es par des r\'eflexions.
  {C}hapitre {VI}: syst\`emes de racines}.
\newblock Actualit\'es Scientifiques et Industrielles, No. 1337. Hermann,
  Paris, 1968.

\bibitem{CARRELL}
James~B. Carrell.
\newblock Torus actions and cohomology.
\newblock In {\em Algebraic quotients. {T}orus actions and cohomology. {T}he
  adjoint representation and the adjoint action}, volume 131 of {\em
  Encyclopaedia Math. Sci.}, pages 83--158. Springer, Berlin, 2002.

\bibitem{Fra1}
Alberto Franceschini.
\newblock Minimal bandwidth for $\mathbb{C}^*$-actions on generalized
  grassmannians.
\newblock {\em Rendiconti del Circolo Matematico di Palermo Series 2}, pages
  1--20, 2022.

\bibitem{Huy}
Daniel Huybrechts.
\newblock Stability structures on {L}ie algebras, after {K}ontsevich and
  {S}oibelman.
\newblock May 2009.

\bibitem{kon}
Jerzy Konarski.
\newblock The {B}-{B} decomposition via {S}umihiro's theorem.
\newblock {\em J. Algebra}, 182(1):45--51, 1996.

\bibitem{LM}
Joseph~M. Landsberg and Laurent Manivel.
\newblock On the projective geometry of rational homogeneous varieties.
\newblock {\em Comment. Math. Helv.}, 78(1):65--100, 2003.

\bibitem{MMW}
Mateusz Micha{\l}ek, Leonid Monin, and Jaros{\l}aw~A. Wi\'sniewski.
\newblock Maximum likelihood degree and space of orbits of a
  $\mathbb{C}^*$-action.
\newblock {\em SIAM J. Appl. Algebra Geometry}, 1(5):60--85, 2021.

\bibitem{WORS5}
Gianluca Occhetta, Eleonora~A. Romano, Luis~E Sol{\'a}~Conde, and
  Jaros{\l}aw~A. Wi{\'s}niewski.
\newblock Rational homogeneous spaces as geometric realizations of birational
  transformations.
\newblock {\em Preprint ArXiv:2112.15130}, 2021.

\bibitem{WORS1}
Gianluca Occhetta, Eleonora~A. Romano, Luis~E. Sol{\'a}~Conde, and
  Jaros{\l}aw~A. Wi{\'s}niewski.
\newblock Small bandwidth $\mathbb{C}^*$-actions and birational geometry.
\newblock {\em To appear in J. Algebraic Geom.}, 2022.

\bibitem{WORS3}
Gianluca Occhetta, Eleonora~A. Romano, Luis~E. Sol{\'a}~Conde, and
  Jaros{\l}aw~A. Wi{\'s}niewski.
\newblock Small modifications of {M}ori dream spaces arising from
  $\mathbb{C}^*$-actions.
\newblock {\em Eur. J. Math.}, pages 1--33, 2022.

\bibitem{OSS}
Christian Okonek, Michael Schneider, and Heinz Spindler.
\newblock {\em Vector bundles on complex projective spaces}.
\newblock Progress in Mathematics, 3. Birkh\"auser, Boston, Mass., 1980.

\bibitem{ReidToric}
Miles {Reid}.
\newblock {Decomposition of toric morphisms.}
\newblock {Arithmetic and geometry, Pap. dedic. I. R. Shafarevich, Vol. II:
  Geometry, Prog. Math. 36, 395-418 (1983).}, 1983.

\bibitem{Spr73}
Tonny~A. Springer.
\newblock {\em Jordan algebras and algebraic groups}.
\newblock Classics in Mathematics. Springer-Verlag, Berlin, 1998.
\newblock Reprint of the 1973 edition.

\bibitem{Tev05}
Evgueni~A. Tevelev.
\newblock {\em Projective duality and homogeneous spaces}, volume 133 of {\em
  Encyclopaedia of Mathematical Sciences}.
\newblock Springer-Verlag, Berlin, 2005.

\bibitem{Thaddeus1996}
Michael Thaddeus.
\newblock Geometric invariant theory and flips.
\newblock {\em J. Amer. Math. Soc.}, 9(3):691--723, 1996.

\bibitem{Thaddeus}
Michael Thaddeus.
\newblock Complete collineations revisited.
\newblock {\em Math. Ann.}, 315(3):469--495, 1999.

\bibitem{Wlodarczyk}
Jaros{\l}aw W{\l}odarczyk.
\newblock Birational cobordisms and factorization of birational maps.
\newblock {\em J. Algebraic Geom.}, 9(3):425--449, 2000.

\end{thebibliography}
\end{document}